\documentclass[11pt, twoside]{article}
\usepackage{amsfonts,amssymb,amsmath,amsthm}
\usepackage{authblk}
\usepackage{graphicx}
\usepackage[all]{xy}
\usepackage{tikz}
\usepackage{changepage}
\usepackage{psfrag,xmpmulti,amscd,color,pstricks, import,enumerate}
\usepackage[normalem]{ulem}

 \divide\oddsidemargin by 2
\newtheorem{thm}{Theorem}[section]
\newtheorem{cor}[thm]{Corollary}
\newtheorem{lem}[thm]{Lemma}
\newtheorem{prop}[thm]{Proposition}
\newtheorem{qn}[thm]{Question}

\newtheorem*{thmA}{Theorem A}
\newtheorem*{thmB}{Theorem B}

\newtheorem*{thmD}{Theorem D}
\newtheorem*{thmC}{Theorem C}
\newtheorem*{thmE}{Theorem E}
\newtheorem*{thmF}{Theorem F}

\theoremstyle{definition}
\newtheorem{defn}[thm]{Definition}

\newtheorem*{rem*}{Remark}
\newtheorem*{rems*}{Remarks}
\newtheorem{ex}[thm]{Example}
\newtheorem*{ex*}{Example}

\numberwithin{equation}{section}

\definecolor{OrangeRed}{cmyk}{0,0.6,1,0}            
\definecolor{DarkBlue}{cmyk}{1,1,0,0.20}
\definecolor{DarkGreen}{cmyk}{1,0,0.6,0.2}
\definecolor{myblue}{rgb}{0.66,0.78,1.00}
\definecolor{Violet}{cmyk}{0.79,0.88,0,0}
\definecolor{Lavender}{cmyk}{0,0.48,0,0}
\definecolor{purpleheart}{rgb}{0.41, 0.21, 0.61}
\definecolor{brick}{cmyk}{0,0.8,0.3,0.5}

\newcommand{\dist}{\operatorname{dist}}

\newcommand{\C}{{\mathbb C}}

\newcommand{\D}{{\mathbb D}}

\newcommand{\N}{{\mathbb N}}

\newcommand{\ra}{\rightarrow}
\newcommand{\lran}{\underset{n\to\infty}{\longrightarrow}}

\newcommand{\ov}{\overline}

\renewcommand{\epsilon}{\varepsilon}
\renewcommand{\phi}{\varphi}

\newcommand{\leucl}{{\ell_{\operatorname{Eucl}}}}

\newcommand{\Rad}{\operatorname{Rad}}

\newcommand{\AP}{\operatorname{AP}}

\vspace{5cm}

\title{Boundary dynamics for holomorphic sequences, non-autonomous dynamical systems and wandering domains}

\author[1]{\hspace{1.7cm}Anna Miriam Benini \thanks{Partially supported by Gruppo Nazionale per l'Analisi Matematica, la Probabilit\`a e le loro Applicazioni (GNAMPA, INdAM) and by  PRIN 2017, Real and Complex Manifolds: Topology, Geometry and holomorphic dynamics.} \thanks{Partially supported by MSRI, Berkeley.}}
\author[4]{Vasiliki Evdoridou \textsuperscript{\textdagger} \thanks{ Supported by the EPSRC grant EP/R010560/1.}}
\author[2,3]{N\'uria Fagella \textsuperscript{\textdagger} \thanks{Partially supported by the Spanish grant MTM2017-86795-C3-3-P and  PID2020-118281GB-C32, and the Catalan grants 2017SGR1374 and ICREA Academia 2020. }}
\author[4]{\hspace{2cm} Philip J. Rippon\textsuperscript{ \textdaggerdbl }}
\author[4]{Gwyneth M. Stallard\textsuperscript{ \textdaggerdbl }}
\affil[1]{\small Dep. of Mathematical, Physical and Computer Sciences, Universit\`a di Parma, Italy.}
\affil[2]{\small Dep. de Matem\`atiques i Inform\`atica, Universitat de Barcelona, Catalonia, Spain.}
\affil[3]{\small Centre de Recerca Matemàtica, Bellaterra, Catalonia, Spain.}
\affil[4]{\small School of Mathematics and Statistics, The Open University, Milton Keynes, UK.}
\date{\today}

\begin{document}
\maketitle

\begin{abstract}

There are many classical results, related to the Denjoy--Wolff Theorem, concerning the relationship between orbits of interior points and orbits of boundary points under iterates of holomorphic self-maps of the unit disc. Here, for the first time, we address such questions in the very general setting of sequences $(F_n)$ of holomorphic maps between simply connected domains. We show that, while some classical results can be generalised, with an interesting dependence on the geometry of the domains, a much richer variety of behaviours is possible. Some of our results are new even in the classical setting.

Our methods apply in particular to non-autonomous dynamical systems, when $(F_n)$ are forward compositions of holomorphic maps, and to the study of wandering domains in holomorphic dynamics.

The proofs use techniques from geometric function theory, measure theory and ergodic theory, and the construction of examples involves a `weak independence' version of the second Borel--Cantelli lemma and the concept from ergodic theory of `shrinking targets'.

\end{abstract}

\section{Introduction}\label{intro}
We consider sequences of holomorphic maps $F_n: U \to U_n, n\in\N,$ between proper, simply connected domains, and prove results concerning the relationship between sequences $(F_n(z))$, where~$z$ is an interior point of $U$ and $(F_n(\zeta))$, where $\zeta$ lies in the {\it boundary} of $U$. In general, the maps $F_n$ are not defined on the boundary of~$U$ but we shall introduce a `radial extension' of $F_n$ which in many cases exists at almost all boundary points, with respect to harmonic measure -- in this situation we say that $F_n$ has a `full radial extension' to $\partial U$; see Section~\ref{prelims} for the details. For simplicity, we use the notation $F_n$ to refer also to the extension of our original map at these boundary points.

In this general setting, we describe the sequence of images $(F_n(z)$), where $z\in U$ or $z\in \partial U$, as an {\it orbit}. There are several natural questions about such orbits: does the behaviour of interior orbits determine the behaviour of boundary orbits (and vice-versa)? If all interior orbits approach the boundary, then do they approach the orbit of a particular boundary point and, if so, how large is the set of boundary points with similar limiting behaviour?

An important case of our general setting is a non-autonomous dynamical system in which $F_n=f_n\circ \cdots\circ f_1$ is a {\it forward composition} of holomorphic maps $f_n: U_{n-1} \to U_n$, $n\in\N$, where $U_0=U$. Systems of forward compositions of holomorphic maps have been studied extensively in recent years in a variety of contexts; see, for example, \cite{FS,Comerford,Sester,Rempe-Urb,ZL}.

A particularly well-studied special case of our general setting is when $U$ is a simply connected {\it wandering domain} of an entire or meromorphic map $f$ and $F_n=f^n$ is the $n$-th iterate of $f$; see, for example, \cite{bfjk,RS-boundaries, bishopwd, marshi, oscillating,BEFRS,MRW,EGP} for many results on such wandering domains. In fact, obtaining a better understanding of the orbits of boundary points of wandering domains was the original motivation for this work.

Another important special case is the following result, where $F_n$ are the iterates~$f^n$ of a holomorphic self-map~$f$ of the unit disc $\D$; see \cite[page~79]{Carleson-Gamelin}, and also \cite{Abate-Chr} for many references to this classical result.

\begin{thm}[Denjoy--Wolff Theorem]\label{thm:DW}
Let $f: \D \to \D$ be holomorphic. Provided $f$ is not a rotation about a point in $\D$, there exists a unique point $p\in\overline{\D}$ (the Denjoy--Wolff point) such that $f^n(z)\to p$ as $n\to\infty$ for all $z\in \D$.
\end{thm}
In this unit disc setting, there is a remarkable dichotomy concerning orbits of boundary points, which originates in the work of Neuwirth \cite{Neuwirth}, Aaronson \cite{aaronson}, and Doering and Ma\~n\'e \cite{doering-mane}. It concerns the iterates of an {\it inner function} $f$; that is, a holomorphic self-map of the unit disc $\D$ whose radial extension maps $\partial\D$ to $\partial\D$, apart from a set of measure zero. Our statement is a modification of the original  dichotomy given in \cite{aaronson} and \cite{doering-mane}, more suited to our present purposes.

\begin{thm}[ADM dichotomy]\label{thm:DM}
Let $f: \D \to \D$ be an inner function with a Denjoy--Wolff point $p\in\overline \D$.
\begin{enumerate}
\item[\rm(a)]If \;$\sum_{n\geq 0}(1-|f^n(z)|)<\infty$ for some $z\in\D$,
then $p\in\partial \D$ and $\lim_{n\to \infty}f^n(\zeta)=p$ for almost every $\zeta \in \partial \D$.

\item[\rm(b)] If \;$\sum_{n\geq 0}(1-|f^n(z)|)=\infty$ for some $z\in\D$, then the iterates $f^n(\zeta), n\in\N,$ are dense in $\partial \D$ for almost every $\zeta \in \partial \D$.
\end{enumerate}
\end{thm}
Thus in this setting if interior orbits tend towards the boundary sufficiently quickly, then almost all boundary orbits tend to the Denjoy--Wolff point, whereas almost no boundary orbits do this if interior orbits tend towards the boundary more slowly. This dichotomy is our main inspiration for obtaining results in the general setting of holomorphic sequences and non-autonomous dynamical systems, with many applications, including to boundary orbits of simply connected wandering domains.

Theorem~\ref{thm:DM}, part~(a) is in \cite[Section 3]{aaronson} and \cite[Theorem~4.1]{doering-mane}, and it was generalised to iterates of holomorphic self-maps of $\mathbb{D}$ in \cite[Theorem~4.2]{BMS}. The proof of Theorem~\ref{thm:DM}, part~(b) uses various results from ergodic theory; see Section~\ref{ADMproof}.

Before discussing generalisations of Theorem~\ref{thm:DM}, we point out that in our general setting orbits of interior points need not have common limiting behaviour; see Example~\ref{Mercer}. However, if the orbit of one interior point converges to the boundary (in any manner), with respect to the Euclidean distance, $\operatorname{dist}(.,.)$, then the orbits of all interior points converge to the boundary and, moreover, they all have a common limiting behaviour.


\begin{thmA}\label{DWthm}
Let $F_n: U \to U_n, n\in\N,$ be a sequence of holomorphic maps between simply connected domains. If there exists $z_0 \in U$ such that
\[
\dist(F_n(z_0), \partial U_n) \to 0 \;\text{ as } n \to \infty,
\]
then, for all $z \in U$,
\[
|F_n(z) - F_n(z_0)| \to 0 \;\text{ as } n \to \infty.
\]
\end{thmA}
%

In Section~\ref{sec:boundaryclassification}, we classify all the possible types of behaviour of interior orbits in relation to convergence to the boundary, and we prove Theorem~A.

In view of Theorem~A, we can make the following definition.

\begin{defn}\label{DWset}
Let $F_n: U \to U_n, n\in\N,$ be a sequence of holomorphic maps between simply connected domains, for which interior orbits converge to the boundary. The {\it Denjoy--Wolff set} of $(F_n)$ consists of those $\zeta \in \partial U$ such that, for all $z \in U$,
\[
|F_n(\zeta) - F_n(z)| \to 0\; \text{ as } n \to \infty.
\]
\end{defn}
\begin{rems*}
1.\; In our general setting the orbits of points $\zeta\in \partial U$ need not lie on the boundaries of the sets $U_n$; they need only satisfy $F_n(\zeta)\in \ov{U_n}$, for $n\in\N$.

2.\; In the case of iterates of a holomorphic self-map $f$ of the unit disc, if the Denjoy--Wolff point~$p$ of~$f$ is on $\partial \D$, then~$f$ is defined at~$p$ via its radial limit, and~$p$ is a fixed point of~$f$ which belongs to the Denjoy--Wolff set. Theorem~\ref{thm:DM} shows that in this situation the Denjoy--Wolff set has either full measure or measure zero.
\end{rems*}

All our results about the boundary orbits of a sequence of holomorphic maps between simply connected domains are stated in terms of the size of the Denjoy--Wolff set.

In Section~\ref{sec:fast},  we show that part~(a) of Theorem~\ref{thm:DM} does indeed have an analogue in our general setting. Interestingly, the result that we obtain depends on the geometry of the boundaries of the domains~$(U_n)$. In the case that all the domains are the unit disc, we obtain a direct generalisation of the result for iterates of inner functions.

\begin{thmB} \label{thm:intro 2}
Let $F_n, n\in \N,$  be a sequence of holomorphic self-maps of $\mathbb{D}$ and suppose that there exists $z_0 \in \mathbb{D}$ such that
\begin{equation}\label{ThmB:ineq}
\sum_{n=0}^{\infty} \left(1-|F_n(z_0)|\right) <\infty.
\end{equation}
Then almost all points in $\partial \D$ belong to the Denjoy--Wolff set of $(F_n)$.
\end{thmB}

In the setting of general simply connected domains~$U$ and $U_n$ instead of the unit disc, an analogue of the condition in Theorem~B would be that

\begin{equation} \label{eq:nopower}
\sum_{n=0}^{\infty} \operatorname{dist}(F_n(z_0),\partial U_n)<\infty.
\end{equation}

However, the less smooth the boundaries of the domains $U_n$ are, the stronger the hypothesis we need. Without any smoothness assumptions, we can prove the following.

\begin{thmC}
Let $F_n:U \to U_n, n\in\N,$ be a sequence of holomorphic maps between simply connected domains, each with a full radial extension to $\partial U$, and suppose  that there exists $z_0 \in U$ such that
\begin{equation} \label{eq:power}
\sum_{n=0}^{\infty} \operatorname{dist}(F_n(z_0),\partial U_n)^{1/2}<\infty.
\end{equation}
Then almost all points in  $\partial U$ belong to the Denjoy--Wolff set of~$(F_n)$.
\end{thmC}

\begin{rems*}
1.\;Condition \eqref{eq:power} in Theorem~C is sharp. Indeed, an example in Section~\ref{sec:fast} shows that the power $1/2$ in~\eqref{eq:power} cannot be increased.

2.\;A special case of Theorem~C is when~$(U_n)$ is an orbit of simply connected wandering domains of a transcendental entire or meromorphic function~$f$. In this case $F_n=f^n$ always has a continuous extension to the boundary, except perhaps at~$\infty$.
\end{rems*}

Despite this successful generalisation of the first part of the ADM dichotomy (Theorem~\ref{thm:DM}, part~(a)), the second part of the dichotomy fails in the general setting of sequences of functions, and even in the setting of non-autonomous dynamical systems; indeed, under the same assumption that interior orbits converge to the boundary slowly, very different types of behaviour can occur on the boundary itself; we illustrate this with examples in Section~\ref{sect:examples}.

In Section 5, we give an example which shows, perhaps surprisingly, that it is possible for the orbits of all interior points to converge to the boundary with the same limiting behaviour and yet the Denjoy--Wolff set is {\it empty}. (This cannot occur in the classical case since the Denjoy--Wolff point always belongs to the Denjoy--Wolff set.)

\begin{thmD}[Empty Denjoy--Wolff set]\label{thm:D-W analogue}
Let $(a_n)$ be an increasing sequence in $[0,1)$ such that $a_0=0$ and $\lim_{n \to \infty}a_n=1$, and let $(\lambda_n)$ be any sequence in $\partial \mathbb{D}$. Then, for all $z\in\D$,
\[
F_n(z)= \frac{\lambda_nz+a_n}{1+\lambda_n a_n z}\to 1\;\;\text{as }n\to \infty.
\]

If, in addition, $\sum_{n\geq 0}(1-F_n(0))=\sum_{n\geq 0}(1-a_n)=\infty,$ then there exists $(\lambda_n)$ in $\partial \D$ such that no point of $\partial \mathbb{D}$ (even the point~1) has an orbit under $F_n$ that converges to~1; in other words, the Denjoy--Wolff set of $(F_n)$ is empty.
\end{thmD}

\begin{rems*}
1.\;If $\sum_{n\geq 0}(1-a_n)<\infty$ in Theorem~D, then it follows from Theorem B that almost all points in $\partial \D$ are in the Denjoy--Wolff set of $(F_n)$.

2.\;Each of the functions $F_n$ in Theorem D can in fact be written as a forward composition of M\"obius maps $f_n$ with $F_n = f_n \circ F_{n-1}$, so this is a non-autonomous example. It remains an open question as to whether this behaviour can occur for wandering domains.
\end{rems*}

So far, we have considered the effect of the behaviour of interior orbits on the behaviour of boundary orbits. In Section 6, we consider the converse problem. We show that if a sufficiently large set of boundary points have orbits with common long-term behaviour then this set must be in the Denjoy--Wolff set; that is, the orbits of all interior points must have the same limiting behaviour as the orbits of this set of boundary points. This is new even for iterates of inner functions.

\begin{thmE}\label{thm:Pom-analogue}
Let $F_n:U \to U_n, n\in\N,$ be a sequence of holomorphic maps between simply connected domains, each with a full radial extension, and suppose that there exists $\zeta_0\in \partial U$ such that $\dist(F_n(\zeta_0),\partial U_n) \to 0$ as $n \to \infty$. If
\[
L_0 := \{\zeta \in \partial U : |F_n(\zeta)- F_n(\zeta_0)| \to 0\;\; \text{as } n \to \infty\}
\]
has positive harmonic measure with respect to~$U$, then for all $z\in U$ we have
\[
|F_n(z)- F_n(\zeta_0)| \to 0\quad\text{as }n\to\infty.
\]
In particular, $L_0$ is in the Denjoy--Wolff set of $(F_n)$.
\end{thmE}
It follows that at most one such set of positive harmonic measure may exist.

For the later sections of the paper we restrict ourselves to the setting of non-autonomous dynamical systems, where each map $F_n$ is the forward composition of holomorphic maps $f_n: U_{n-1} \to U_n$ between simply connected domains. In this setting, we give a classification of possible internal dynamics in terms of the limiting behaviour of the hyperbolic distance between orbits of pairs of points, which extends our classification of wandering domains in~\cite{BEFRS}. This classification is in Section~\ref{sec:thmF}, together with an example which shows that the classification does not hold in general for {\it sequences} of holomorphic maps. One of the three possibilities in our classification is that the maps are {\it contracting}, defined as follows. Here we use the notation $\dist_U$ to denote the {\it hyperbolic distance} in a simply connected domain~$U$.

\begin{defn}\label{cont}
A sequence of maps $F_n: U \to U_n, n\in\N,$ is {\it contracting} if
\[
\dist_{U_n}(F_n(z), F_n(z'))\to  0 \;\text{ as }n\to\infty,\quad \mbox{for all } z,z'\in U.
\]
\end{defn}

For contracting  forward compositions of holomorphic maps  which preserve boundaries, we prove the strong result that the Denjoy--Wolff set has either full or zero measure.

\begin{thmF}
Let $f_n:U_{n-1} \to U_n$, $n\in\N$, be a sequence of holomorphic maps between simply connected domains, each with a full radial extension, and let $F_n: U_0 \to U_n$ be defined by
\[
F_n = f_n \circ \cdots \circ f_1,\quad\text{for } n\in\N,
\]
with the additional assumption that
\begin{equation}\label{fn-bdry-prop}
f_n(\zeta) \in \partial U_n,\;\text{ for almost all } \zeta \in \partial U_{n-1}.
\end{equation}
If $(F_n)$ is contracting and interior orbits converge to the boundary, then the Denjoy--Wolff set has either full or zero harmonic measure with respect to $U_0$.
\end{thmF}
\begin{rems*}
1.\;In Section~\ref{prelims}, we show that the composition of functions with full radial extensions also has a full radial extension; this is the case, for example, when $U_0$ is a wandering domain of an entire function~$f$ with $F_n=f^n$ for $n\in\N$ or when all the functions $f_n$, $n\in\N$, are inner functions.

2.\;Several results in this paper, including Theorem~F, can be applied when the domains $(U_n)$ are multiply connected and we are concerned with orbits converging to the `outer boundary'; that is, the boundary of the topological hull of $U_n$. This approach is used in \cite[Theorem~1.4]{Gustavo}, where a result analogous to Theorem~A is given in the setting of dynamical behaviour near the outer boundary of multiply connected wandering domains of meromorphic functions.
\end{rems*}

In Section~\ref{sec:thmF}, we prove a general result, Theorem~\ref{thm:Pomgen}, that has Theorem F as a corollary. Our proof of Theorem~\ref{thm:Pomgen} uses an ergodic theory result of Pommerenke and generalises one aspect of it. We also give an example to show that Theorem~F does not hold in general if the forward compositions are not contracting.

In Section~\ref{sect:examples}, we give a number of examples of forward compositions of self-maps of the unit disc, related to part~(b) of the ADM-dichotomy. Whether  the interior and boundary dynamics demonstrated in these examples can  be replicated by examples of simply connected wandering domains is an open question, and will be the subject of further investigation. Our final example, Example~\ref{ex:z2}, is a contracting sequence; for this example, we are able to show that the Denjoy--Wolff set has measure zero and, in fact, that orbits of almost all boundary points are dense in $\partial \D$. This raises the question of whether an analogue of part~(b) of the ADM dichotomy may exist in the setting of contracting forward compositions. We discuss this in Section~\ref{ADMproof} and also in the final section of the paper on open questions.

In Section~\ref{sec:spherical}, we discuss briefly the extent to which our results hold if we replace the Euclidean metric as a measure of proximity by the spherical metric; the latter metric has the advantage of including those points whose orbits tend to $\infty$ but at the same time we potentially lose control of some of the more subtle behaviour amongst those orbits that do tend to $\infty$.

The proofs of Theorems A to F and of our examples use a variety of techniques from geometric function theory, harmonic analysis and measure theory including contraction of the hyperbolic metric, L\"owner's lemma, the Milloux--Schmidt inequality, the Severini--Egorov theorem, and a version of the second Borel--Cantelli lemma associated with the ergodic theory concept of `shrinking targets'.

\section{Preliminary results}\label{prelims}
In this section we record a number of properties of certain boundary extensions of holomorphic maps between simply connected domains and the relationship of these boundary extensions to the harmonic measure of sets on the boundaries of these domains. In doing so, we use a number of classical results on the boundary behaviour of holomorphic maps, which can mostly be found in~\cite{Pommerenke}.

In this section only, we denote such boundary extensions of a holomorphic map~$f$ by~$f^*$ for clarity; elsewhere in the paper we drop the * notation for simplicity.

\subsection{Boundary extensions}\label{sect:boundary maps}

Recall that our general setting is that of sequences of holomorphic maps
\[
F_n : U \longrightarrow U_n,\quad n\in\N,
\]
where $U$ and $U_n, n\in\N,$ are simply connected domains, always assumed to be proper subsets of $\C$ and not necessarily distinct, and that in some cases the functions $F_n$ are defined by forward composition as $F_n=f_n\circ \cdots \circ f_1$, where each $f_n: U_{n-1}\to U_n$ is holomorphic and $U_0=U$.

In the wandering domains context, where $F_n=f^n$ with $f$ entire, and $f: U_{n-1} \to U_{n}$,  the maps are naturally defined and continuous on the whole of $\partial U_0$ (except maybe at the point at infinity). In the inner functions context, however, where $U_n=\D$, for $n\ge 0$, each $F_n$ extends in the sense of radial limits and even nontangential limits to a full measure subset of $\partial \D$ by theorems of Fatou and Lindel\"of; see \cite[p.19]{Coll-Loh}, for example. We now consider whether such boundary extensions exist in the general setting of this paper.

We shall suppose quite generally that~$U$ and~$V$ are simply connected proper subdomains of~$\C$ and $f:U\to V$ is holomorphic. We wish to define the concept of a `radial extension' of~$f$ to boundary points of $U$, wherever this extension makes sense.

A simple case is obtained when $U=\D$, that is, $f:\D\to V$ is a holomorphic map (not necessarily conformal). If~$V$ is bounded, then the radial limit at $\zeta\in\partial\D$ exists if and only if the non-tangential limit at $\zeta$ exists and this occurs if and only the limit exists along any curve in~$\D$ ending at~$\zeta$, by the theorem of Lindel\"of mentioned above; moreover, all these limits coincide. In the general case, since the complement of~$V$ contains an unbounded continuum, the function
\[
g(z)=\sqrt{f(z)-w_0}
\]
where $w_0\notin V$ and a suitable branch of the square root is taken, is holomorphic in $\D$ and $g(\D)$ omits a disc, with centre $w_1$ say; thus the function $z\mapsto (g(z)-w_1)^{-1}$ is holomorphic and bounded in $\D$, so we see that the relationship between radial limits, nontangential limits and limits along curves remains true for such~$f$. Therefore it makes sense to define the {\em radial  extension} of such an~$f$ as the map
\[
f^*(\zeta)=\lim_{r\ra1}f(r \zeta),\;\text{ for }\zeta\in\Rad(f),
\]
where $\Rad(f)$ denotes the subset of $\partial \D$ at which~$f$ has a finite radial limit. The function $f^*$ assigns to each point of $\partial \D$ its radial limit, whenever it exists, and takes values in the closure of $f(\D)$. Note that even when $f$ has a nontangential limit at $\zeta\in\partial \D$, the function~$f$ may fail to converge to that limit along curves in $\D$ ending at $\zeta$ that are not nontangential.

Moving to the case where $f:U\to V$, with $U$ and $V$ being proper simply connected domains, we introduce a Riemann map $\phi:\D\to U$. By the discussion above,~$\varphi$ has radial (and equal nontangential) limits almost everywhere on $\partial\D$ and in fact $\partial \D\setminus \Rad(\varphi)$ has capacity zero by a theorem of Beurling \cite[Theorem~9.19]{Pommerenke}.

A point $p\in\partial U$ is called {\em accessible} from~$U$ if there exists an arc $\gamma:[0,1)\to U$ such that $\gamma(t)\to p$ as $t\to 1$. We denote the set of accessible boundary points of~$U$ by $\AP(U)$ (there seems to be no standard notation for this set). Every homotopy class of curves $\gamma\subset U$ with their endpoints fixed and one endpoint at an accessible boundary point~$p$  is called an {\em access} to~$p$ from~$U$. It follows from another theorem of Lindel\"of \cite[page~8]{Carleson-Gamelin} that there is a bijection between the set $\Rad(\varphi)$ and the set of accesses to boundary points of~$U$, induced by~$\phi$; see \cite[Correspondence Theorem]{BFJKaccesses}, for example.

In defining the concept of a radial extension of $f:U\to V$, our aim is that the extension at almost all accessible boundary points of~$U$ is consistent with the limiting behaviour of~$f$ along curves to $\partial U$ through all possible accesses to such points.
\begin{defn}[Radial extensions] \label{defn:boundary extension general}
Let $f:U\to V$ be a holomorphic map between simply connected domains, let $p\in \AP(U)$,  and let $\phi:\D\to U$ be a Riemann map. Then we say that
\begin{itemize}
\item[(a)]
a curve in~$U$ approaches~$p$ {\em radially} if it is the image of a curve in $\D$ that approaches a point of $\partial \D$ nontangentially;
\item[(b)]
$f$ has a {\em radial extension} to~$p$ with value $f^*(p)$ if for every arc $\gamma(t), t\in [0,1),$ in~$U$ such that $\gamma(t)\to p$ radially as $t\to 1$, the image curve  $f(\gamma(t))\to f^*(p)$ as $t\to 1$;
\item[(c)]
$f$ has a {\em full radial extension} (to $\partial U$) if the domain of $f^*$ includes a subset of $\AP(U)$ of full harmonic measure with respect to~$U$.
\end{itemize}
\end{defn}
We make two remarks about Definition~\ref{defn:boundary extension general}.
\begin{rems*}
1.\;We focus on the limiting behaviour of~$f$ only on curves in~$U$ that approach a boundary point~$p$ {\em radially} in order to be consistent with the definition of radial extension of a holomorphic map $f:\D\to V$ in terms of nontangential limits, given above.

2.\;The existence of a full radial extension of~$f$ implies that (up to a set of harmonic measure zero)~$f$ maps accesses to a point $p\in\partial U$ to accesses to $f^*(p)\in\partial V$, unless $f^*(p)$ belongs to~$V$, which is a priori also possible.
\end{rems*}

In general, the radial extension of a conformal map $\varphi:\D\to U$ is not an injective map, since boundary points of~$U$ can have several (possibly uncountably many) accesses from~$U$. Nevertheless $(\varphi^*)^{-1}(p)$ is always a set of measure zero, by the F. and M. Riesz theorem \cite{Duren}. Boundary points with a single access from~$U$ are called {\em uniquely accessible} and domains for which all their accessible boundary points are uniquely accessible are said to satisfy the  \emph{unique accessibility property}. In this case the radial extension  of~$\phi$ is injective (and therefore bijective).

There are some situations in which full radial extensions exist in a trivial way; in particular, this is the case when~$U$ and~$V$ are Fatou components of a holomorphic map or when~$f$ is an inner function, since $U=V=\D$ in this case.  More generally, we have the following.

\begin{prop}[Simple cases] \label{simplecases}
Let $f:U\to V$ be a holomorphic map between simply connected domains. If
\begin{enumerate}
\setlength{\itemsep}{0.05cm}
\item[\rm (a)] $f$ has a continuous extension to $\partial U$, or
\item[\rm (b)]  $U$ has the unique accessibility property,
\end{enumerate}
then $f$ admits a full radial extension $f^*$. In case (a), we have $f^*:=f|_{\partial U}$.
\end{prop}

It is interesting to ask whether if $f:U\to V$ is a holomorphic map between simply connected domains and~$f$ has a full radial extension to $\partial U$, then this implies that at almost all $\zeta\in\partial U$, with respect to harmonic measure, either~$f$ has a continuous extension to~$\zeta$ or the point~$\zeta$ is uniquely accessible from within $U$?
It is straightforward to construct examples where both types of boundary behaviour occur in sets of positive harmonic measure.

\begin{proof}[Proof of Proposition~\ref{simplecases}]
If $f$ is continuous in $\overline U$, then the conclusion  that~$f$ has a full radial extension is trivial.

For case (b) let $\phi_U, \phi_V$ be Riemann maps from $\D$ to $U, V$ respectively and let $g:\D\ra\D$ be defined as $g:=\phi_V^{-1}\circ f\circ \phi_U$.

If $U$ has the unique accessibility property, then $\varphi_U^*$ is injective,  so its inverse is well defined on $\phi^*(\partial \D)$, which has full harmonic measure in~$U$ because $\phi_U$ is the Riemann map. Hence we can define
\[
f^*:=(\varphi_V\circ g)^*\circ(\varphi_U^*)^{-1}
\]
on the set $E=\AP(U)\cap\varphi_U^*(\Rad(\varphi_V\circ g))$, which is of full harmonic measure in~$U$ since it is the intersection of two sets of full harmonic measure. With this definition, if $\gamma(t)$ in~$U$ converges radially to $p\in E$, then $\varphi_U^{-1}(\gamma(t))$ converges nontangentially to a point $\zeta\in \Rad(\varphi_U)\cap \Rad(\varphi_V\circ g)$, so~$f$ has a full radial extension.
\end{proof}
We also point out to what extent the radial extension property is preserved under forward composition.

\begin{prop}\label{composing-radial extns}
Let $f_1:U_0\to U_1$ and $f_2:U_1 \to U_2$ be holomorphic maps between simply connected domains, with radial extensions $f_1^*$ to $E_0\subset \AP(U_0)$ and $f_2^*$ to $E_1\subset \AP(U_1)$, respectively.

Then $f_2\circ f_1:U_0\to U_2$ has a radial extension to almost all points of $E_0 \setminus (f_1^*)^{-1}(\partial U_1\setminus E_1)$.
\end{prop}
\begin{proof}
For  $\zeta_0\in E_0\setminus (f_1^*)^{-1}(\partial U_1\setminus E_1)$,
let $\gamma_0(t), 0\le t<1,$ be a path in~$U_0$ that converges radially to~$\zeta_0$. Then, by our assumption about $\zeta_0$, we have $\gamma_1(t):=f_1(\gamma_0(t))\to \zeta_1:=f_1^*(\zeta_0)$ as $t\to 1$; if $\zeta_1\in U_1$ we are done, otherwise  $\zeta_1\in E_1$. Note that $\gamma_1(t)$ may or may not converge radially to $\zeta_1$.

Now let $\hat\gamma_1(t), 0\le t<1,$ be a path in~$U_1$ that does converge radially to~$\zeta_1$ and is homotopic in $U_1$ to $\gamma_1(t)$. Then $f_2(\hat\gamma_1(t))\to \hat\zeta_2:=f_2^*(\zeta_1)$ as $t\to 1$. Suppose, if possible, that $f_2(\gamma_1(t))=f_2(f_1(\gamma_0(t)))\nrightarrow \hat\zeta_2$ as $t\to 1$. Since $\gamma_0(t)$ converges to $\zeta_0$ radially, we know that apart from a subset of $E_0$ of harmonic measure zero $f_2(f_1(\gamma_0(t)))$ converges to a point, $\zeta_2\in \partial U_2$ say, and we can only have $\zeta_2\neq\hat\zeta_2$ for countably many $\zeta_0\in E_0$, by Bagemihl's ambiguous point theorem \cite{Pommerenke}, applied in $U_1$. Therefore, for all $\zeta_0\in E_0\cap (f_1^*)^{-1}(E_1)$ apart from a set of harmonic measure zero, we have $f_2(\gamma_1(t))=f_2(f_1(\gamma_0(t)))\to\hat\zeta_2$, and so $f_2\circ f_1$ has a radial extension to $\zeta_0$,  as required.
\end{proof}

\subsection{Properties of harmonic measure} \label{sec:prop_harmonic}
In order to work with our definition of a radial extension, which involves the harmonic measure of boundary sets of simply connected domains, we first recall certain topological properties of boundary sets, and the behaviour of their images and pre-images under holomorphic mappings. Theorem~\ref{thm:PommSect6.2}, part~(a) is given in \cite[Proposition~6.5]{Pommerenke} and part~(b) is due to Cant\'on, Granados and Pommerenke \cite[Theorem~1]{Canton}.
\begin{thm}[Borel sets on the boundary]\label{thm:PommSect6.2}
Let $f: \D\to \hat\C$ be continuous.
\begin{itemize}
\item[(a)]
The set $\Rad(f)$ is a Borel set and, for any Borel set $B\subset \hat\C$,
\[
(f^*)^{-1}(B)=\{\zeta\in \Rad(f): f^*(\zeta)\in B\}\;\text{ is a Borel set.}
\]
\item[(b)]
Moreover, if $f$ is a conformal map and $A\subset \Rad(f)$ is a Borel set, then
\[
f^*(A)=\{f^*(\zeta): \zeta\in A\}\;\text{ is a Borel set}.
\]
\end{itemize}
\end{thm}
In particular, it follows from Theorem~\ref{thm:PommSect6.2} that the set of accessible points of $\partial U$ is a Borel set, since $\AP(U)$ is the image under a Riemann map of the Borel set $\Rad(f) \subset \partial \D$ (of full measure).

Note that every Borel set is a Souslin set (see \cite[page~132]{Pommerenke}) and hence universally measurable, and in particular measurable with respect to Lebesgue measure on $\partial\D$ and with respect to harmonic measure, which we discuss next.

We shall briefly recall the definition of harmonic measure and some properties of it that we need. Good references for harmonic measure are \cite[Chapter~2]{Conway2}, \cite{GarnettMarshall},  \cite{Pommerenke} and \cite{BracciContrerasDM}.

A standard definition of harmonic measure in the unit disc is in terms of the Poisson integral
\[
\omega(z,A,\D)=\frac{1}{2\pi}\int_A \frac{1-|z|^2}{|\zeta-z|^2}\,|d\zeta|,\quad z\in \D,
\]
where $A\subset \partial\D$ is measurable. The real function~$\omega$ is harmonic in~$\D$ with $0\le \omega(z,A,\D)\le 1$ there, and has radial limit~1 (and indeed nontangential limit~1) at almost every point in~$A$ and radial limit~0 at almost every point in $\partial \D\setminus A$.

It follows that for any such set $A$ we have $\omega(0,A,\D)=\lambda(A)$, where $\lambda$ is normalised Lebesgue measure on $\partial\D$, and more generally  $\omega(z,A,\D)=\lambda(M(A))$, where $M$ is a M\"obius map of $\D$ onto $\D$ taking~$z$ to~0; see \cite[page~85]{Pommerenke}.

We now define harmonic measure in a general simply connected domain $U$ in terms of the pullback under a Riemann map of the normalized Lebesgue measure on the unit circle, following the approach in \cite[Chapter~7]{BracciContrerasDM}.

\begin{defn}[Harmonic measure in a simply connected domain] Let $U\neq\C$ be a simply connected domain, $z\in U$, and let   $\phi:\D\ra U$ be a conformal map such that $\phi(0)=z$.  The \emph{harmonic measure} of a Borel set $A\subset \partial U$ is defined as
$$
\omega(z, A, U)= \lambda((\varphi^*)^{-1}(A)).
$$
(Note that $(\varphi^*)^{-1}(A)$ is a Borel set in $\partial \D$ by Theorem~\ref{thm:PommSect6.2}, part~(a).)
\end{defn}

Since $\lambda$ is invariant under rotation, this definition is independent of the choice of $\phi$ provided it satisfies $\phi(0)=z$.

The harmonic measure $\omega(z, A, U)$ can also be interpreted as the probability of hitting the boundary at a point in~$A$ while following  a Brownian path in~$U$ starting at the point~$z$; see \cite[Section~3.4]{Morters-Peres} and \cite[Chapter~III]{GarnettMarshall}, for example.

By the theorem of Beurling mentioned earlier, the set $\AP(U)$ has full harmonic measure with respect to~$U$. Also, for any Borel set $A\subset \partial U$, the function $\omega(z,A,U)$ is harmonic in~$U$ and takes values in the interval $[0,1]$. Moreover, for almost all points $\zeta\in \AP(U)$ there is a path $\gamma(t)\subset U, 0\le t<1$, the image of a radius in $\D$, such that $\gamma(t)\to \zeta$ and $\omega(\gamma(t),A,U)\to 1_A(\zeta)$ as $t\to 1-$. In this sense, the function $\omega(z,A,U)$ is the solution of the Dirichlet problem in~$U$ with boundary values $\omega(\zeta, A,\D)=1_A(\zeta)$.

\subsection{Radial extensions expand harmonic measure}\label{sect:harmonic measure}
In this section we show that the radial extension of a holomorphic map as defined in Section~\ref{sect:boundary maps} expands the harmonic measure of sets in the boundary. Our result depends on the following version of L\"owner's lemma, which can be found in \cite[Proposition 4.15]{Pommerenke} and \cite[Theorem 7.1.8 and Proposition 7.1.4 part~(4)]{BracciContrerasDM}.
\begin{thm}[L\"owner's lemma]\label{thm:Ransford non continuous} Let~$f$ be a holomorphic self-map of~$\D$, $f^*$ be its radial extension,  and let $S\subset \partial \D$ be a Borel set. Then
\begin{equation}\label{eqtn:Ransford Phil}
\omega(z, (f^*)^{-1}(S), \D)\leq \omega(f(z), S,\D),\;\text{ for } z\in\D.
\end{equation}
\end{thm}
\begin{rem*}
If~$f$ is a M\"obius map of $\D$ onto $\D$, or even an inner function, then equality holds in \eqref{eqtn:Ransford Phil} throughout $\D$; see \cite[Corollary~1.5(b)]{doering-mane}.
\end{rem*}

In our proofs we use the following extension of Theorem~\ref{thm:Ransford non continuous} to simply connected domains.
\begin{thm}[Expansion on the boundary of simply connected domains]\label{Cor:expansion}
Let $f:U\to V$ be a holomorphic map between simply connected domains with full radial extension $f^*$ and let $S\subset \partial V$ be a Borel set. Then $(f^*)^{-1}(S)$ is a Borel set and, for $z\in U$,
\begin{equation}\label{eqtn:Ransford Phil general}
\omega(z, (f^*)^{-1}(S), U)\leq \omega(f(z), S,V).
\end{equation}
In particular, if  $(f^{*})^{-1}(S)$ has positive (respectively, full) harmonic measure with respect to~$U$, then $S$ has positive (respectively, full) harmonic measure with respect to~$V$.
\end{thm}
\begin{proof}
First,  let $\phi_U$ and $\phi_V$ be Riemann maps from $\D$ to $U$ and $V$, respectively. Then $(f^*)^{-1}(S)$ is a Borel set because it is the image under $\phi_U^*$ of the preimage under$(f\circ \phi_U)^*$ of~$S$, apart from a possible set of harmonic measure~0. This follows from our definition of radial extensions and by applying both parts of Theorem~\ref{thm:PommSect6.2}.

Now define
\[
g(w)=\phi_V^{-1}\circ f \circ \phi_U(w),\;\text{ for }w\in \D,
\]
so $g$ is a self-map of $\D$. Then $(\phi_V^*)^{-1}(S)$ is a Borel set, by Theorem~\ref{thm:PommSect6.2}\;(a). Also, by our definition of radial extension, and also the bijection between accesses and non-tangential limits of Riemann maps, mentioned earlier, it is straightforward to check that apart from a set of harmonic measure~0 on $\partial \D$, we have
\[
((f\circ \phi_U)^*)^{-1}(S)\subset (g^*)^{-1}((\phi_V^*)^{-1}(S)).
\]
Therefore, by Theorem~\ref{thm:Ransford non continuous}, we know that, for $w\in\D$,
\begin{equation}\label{omegas}
\omega(w,((f\circ \phi_U)^*)^{-1}(S), \D)\le \omega(w, (g^*)^{-1}((\phi_V^*)^{-1}(S)), \D)\leq \omega(g(w),(\phi_V^*)^{-1}(S),\D).
\end{equation}
Finally, we take $z\in U$ and assume that the Riemann maps have been chosen so that
\[
\phi_U(0)=z\quad\text{and}\quad \phi_V(0)=f(z).
\]
Then $g(0)=0$ and we have, by the definition of harmonic measure,
\[
\omega(0,((f\circ \phi_U)^*)^{-1}(S),\D)=\omega(z, (f^*)^{-1}(S), U)\;\;\text{and}\;\;\omega(0,(\phi_V^*)^{-1}(S),\D)=\omega(f(z), S,V),
\]
so the inequality \eqref{eqtn:Ransford Phil general} follows by applying \eqref{omegas} with $w=0$.
\end{proof}
Combining Theorem~\ref{Cor:expansion} with Proposition~\ref{composing-radial extns}, we obtain the following useful result about composing full radial extensions.
\begin{cor}\label{composing-full-radial-extns}
Let $f_1:U_0\to U_1$ and $f_2:U_1 \to U_2$ be holomorphic maps between simply connected domains, with full radial extensions $f_1^*$  and $f_2^*$, respectively.

Then $f_2\circ f_1:U_0\to U_2$ has a full radial extension given by $f_2^*\circ f_1^*$.
\end{cor}
\begin{proof}
By hypothesis the sets $E_0$ and $E_1$ specified in Proposition~\ref{composing-radial extns} have full harmonic measure with respect to~$U_0$ and~$U_1$, respectively, and we can deduce that $(f_1^*)^{-1}(\partial U_1\setminus E_1)$ must have zero harmonic measure with respect to~$U_0$ by Theorem~\ref{Cor:expansion}. It follows that $E_0 \setminus (f_1^*)^{-1}(\partial U_1\setminus E_1)$ has full harmonic measure with respect to~$U_0$, as required.
\end{proof}

\section{Boundary convergence and proof of Theorem A}\label{sec:boundaryclassification}

Before proving Theorem~A, we give a classification of sequences $F_n: U \to U_n, n\in\N,$ of holomorphic maps between simply connected domains in terms of whether orbits $F_n(z)$ of interior points $z\in U$ converge towards the boundary or not. We gave such a classification in the special case that~$U$ is a wandering domain of a holomorphic map $f$ in \cite[Theorem~C]{BEFRS}. The same proof holds in our general setting and the techniques used also lead to a proof of Theorem~A.

\begin{thm}[Boundary convergence classification]\label{bdry-class}
Let $F_n:U \to U_n$ be a sequence of holomorphic maps between simply connected domains. Then exactly one of the following holds:
\begin{itemize}
\item[\rm(a)]
  $\liminf_{n\to\infty} \operatorname{dist}(F_{n}(z),\partial U_{n})>0$  for all $z\in U$,
    that is, all orbits stay away from the boundary;
\item[\rm(b)] there exists a subsequence $n_k\to \infty$ for which $\lim_{k\to\infty}\operatorname{dist}(F_{n_k}(z),\partial U_{n_k}) = 0$ for all $z\in U$, while for a different subsequence $m_k\to\infty$ we have that
    \[\liminf_{k \to \infty} \operatorname{dist}(F_{m_k}(z),\partial U_{m_k})>0, \quad\text{for }z\in U;\]
\item[\rm(c)] $\lim_{n\to\infty}\operatorname{dist}(F_{n}(z),\partial U_{n})= 0$ for all $z\in U$, that is, all orbits converge to the boundary.
\end{itemize}
\end{thm}
Our proof of this result in \cite{BEFRS} in the special case that~$U$ is a wandering domain of a holomorphic function $f$ with $F_n = f^n$ depends on the following lemma \cite[Lemma~4.1]{BEFRS}, relating hyperbolic density $\rho_U$ to hyperbolic distance in a simply connected domain~$U$. Hence the original proof of Theorem~\ref{bdry-class} for simply connected wandering domains also works in our general setting.

\begin{lem}[Estimate of hyperbolic quantities]\label{lem:hyp density}
Let $U \subset \mathbb{C}$ be a simply connected domain. Then, for all $z,z'\in U$,
\[
\exp(-2\dist_U(z,z'))\le \frac{\rho_U(z')}{\rho_U(z)} \le \exp(2\dist_U(z,z')).
\]
\end{lem}

Theorem~A strengthens part~(c) of the boundary convergence classification by showing that all orbits converge to the boundary with the same limiting behaviour, and its proof is also based on Lemma~\ref{lem:hyp density}. It is useful here to state a more precise result than Theorem~A.
\begin{thm}\label{thm:A+}
Let $F_n: U \to U_n, n\in\N,$ be a sequence of holomorphic maps between simply connected domains and let $z, z_0\in U$ with $d=\dist_U(z,z_0)$. Then
\begin{equation}\label{eucl-hyp}
|F_n(z)-F_n(z_0)|\le 2de^{2d}\dist(F_n(z_0),\partial U_n),\; \text{ for } n\in\N.
\end{equation}
Therefore, if we also have
\[
\dist(F_n(z_0), \partial U_n) \to 0 \;\text{ as } n \to \infty,
\]
then, for all $z \in U$,
\[
|F_n(z)-F_n(z_0)| \to 0 \;\text{ as } n \to \infty.
\]
\end{thm}
\begin{proof}
Let $z,z_0$ and $d$ be as in the statement and let
\[\delta_n:=\dist(F_n(z_0),\partial U_n), \;\text{ for } n\in\N.
\]
By the standard estimates of the hyperbolic density on simply connected domains (see \cite[page~13]{Carleson-Gamelin} or \cite[Theorem 8.6]{BeardonMinda}),
$$
\rho_{U_n}(F_n(z_0))\geq\frac{1}{2\dist(F^n(z_0),\partial U_n)}=\frac{1}{2\delta_n}, \;\text{ for } n\in\N.
$$
Hence, by Lemma~\ref{lem:hyp density}, for any hyperbolic disc $\Delta_n$ of radius $d$ centred at $F_n(z_0)$ we have
\begin{equation}\label{eq:hyp density}
\rho_{U_n}(w)\geq \frac{e^{-2d}}{2\delta_n}, \quad \text{for }w \in \Delta_n, n\in\N.
\end{equation}
By the contracting property of the hyperbolic metric, we have $F_n(z)\in\ov{\Delta_n}$ for $n\in\N$.

Now let $\gamma\subset \Delta_n$ be a hyperbolic geodesic joining $F_n(z)$ to $F_n(z_0)$; then the Euclidean distance apart of these two points is bounded from above by the Euclidean length $\leucl(\gamma)$. Since $\gamma$ is a hyperbolic geodesic, we deduce from \eqref{eq:hyp density} that
$$
d\geq \ell_{U_n}(\gamma)=\int_\gamma \rho_{U_n}(z)\,|dz|\geq \frac{e^{-2d}}{2\delta_n}\leucl(\gamma),
$$
so
$$
|F_n(z)-F_n(z_0)|\le \leucl(\gamma)\leq 2\delta_n d e^{2d},\quad \text{for }n\in \N,
$$
which gives \eqref{eucl-hyp} and completes the proof.
\end{proof}
\begin{rem*}
The matter of convergence to the boundary is somewhat delicate in that it is closely related to the shape of the domains $U_n$, and there may be situations where it is more appropriate to use an alternative definition. For example, if the domains shrink with Euclidean diameters tending to 0, then the Denjoy--Wolff set is automatically the whole of $\partial U$. This geometric problem was discussed briefly in \cite{BEFRS}, where the following observation was made.

Recall that the Euclidean distance of a point~$z$ from the boundary of a hyperbolic domain~$U$ is closely related to the hyperbolic density at the point in the domain. Indeed, when $U_n$ are simply connected domains, the standard estimates for hyperbolic density mentioned above imply that
\[
\dist(F_n(z), \partial U_n) \to 0 \iff \rho_{U_n}(F_n(z)) \to \infty.
\]
The proof of the boundary convergence classification and of Theorem~\ref{thm:A+} use the hyperbolic densities $\rho_{U_n}(F_n(z))$ and, as in \cite[Section~4]{BEFRS}, we point out that these results remain true if we replace $\rho_{U_n}(F_n(z))$ by $a_n\rho_{U_n}(F_n(z))$, where the positive sequence $(a_n)$ is chosen in some way depending on the geometry of the domains $U_n$, and the Euclidean distances within $U_n$ are also scaled appropriately. For example, if the domains $U_n$ are shrinking, then it may make sense to say that $F_n(z)$ converges to the boundary if $a_n \rho_{U_n}(F_n(z)) \to \infty$ as $n \to \infty$ where
$$a_n=\sup_D \{\operatorname{diam}D :  D\; \text{is a disc contained in}\;U_n\}.$$
\end{rem*}

\section{Fast convergence to the boundary}\label{sec:fast}
In this section we prove  a general result, Theorem \ref{thm:general result}, of which Theorem B
and Theorem C are special cases. This  result shows that if orbits of points converge sufficiently quickly to the boundary, then the Denjoy--Wolff set has full measure, thus generalising the ADM dichotomy  Theorem~\ref{thm:DM}, part (a).

\begin{thm}\label{thm:general result}
Let $F_n:U \to U_n, n\in\N,$ be a sequence of holomorphic maps between simply connected domains, each with a full radial extension, and suppose that the domains $U_n, n\in\N,$ all satisfy an $\alpha$\emph{-harmonic measure condition} for some $\frac{1}{2}\leq\alpha\leq1$ and some $C(r), r>0,$ independent of $n$.

Suppose that  there exists $z_0 \in U$ such that
\begin{equation}\label{eq:general1}
 \sum_{n=0}^{\infty} \operatorname{dist}(F_n(z_0), \partial U_n)^{\alpha} < \infty.
\end{equation}
Then, for almost every $\zeta \in \partial U$,
\begin{equation}\label{eq:general2}
|F_n(\zeta)-F_n(z_0)| \to 0 \; \text{ as $n \to \infty$},
\end{equation}
and hence, in view of Theorem~A,  the Denjoy--Wolff set has full measure in $\partial U$.
\end{thm}
\begin{rems*}\label{rem:Thm4.1}
1.\;The estimate \eqref{eucl-hyp} in the statement of Theorem~\ref{thm:A+} shows that condition~\eqref{eq:general1} is independent of the choice of  $z_0\in U$.

2.\;If we replaced \eqref{eq:general1} by the hypothesis that $\operatorname{dist}(F_n(z_0),\partial U_n)\to 0$ as $n\to\infty$, then we would  obtain the same result for  any subsequence $(F_{n_k})$ for which the corresponding series in \eqref{eq:general1} is convergent.
\end{rems*}
We give a formal definition of the $\alpha$-harmonic measure condition at the beginning of Subsection~\ref{geometric}. This is a geometric condition on the boundaries of the domains which determines the rate of convergence required in~\eqref{eq:general1}.

We shall see that a domain with $C^{2}$ boundary, such as the unit disc, satisfies the $\alpha$-harmonic measure condition with $\alpha=1$.   On the other hand, a general  simply connected domain~$U$, under no assumptions on the regularity of  its boundary, satisfies an $\alpha$-harmonic measure condition with $\alpha=1/2$. Hence
\begin{itemize}
\item[(a)] Theorem~\ref{thm:DM}, part (a) is a special case of Theorem \ref{thm:general result}, with $U=U_n= \mathbb{D}$ and $F_n=f^n$, where $f$ is an inner function, for $n \in \mathbb{N}$;
\item[(b)] Theorem B is a  special case of Theorem \ref{thm:general result},   with $U_n= \mathbb{D}$ and $F_n:\D\ra\D$  holomorpic self maps of the disc, for $n \in \mathbb{N}$;
\item[(c)]Theorem C follows from Theorem \ref{thm:general result} if we make no assumptions on the geometric nature of the domains and take  $\alpha=\frac{1}{2}$ in (\ref{eq:general1}).
\end{itemize}

In Section~\ref{cardioid} we give an example which shows that, in one sense, Theorem~\ref{thm:general result} is best possible. For this example, where the domains are all cardioids, the conclusion holds if we take $\alpha = 1/2$ in~\eqref{eq:general1}, but does not hold for any larger value of $\alpha$.

\subsection{Domains with the $\alpha$-harmonic measure condition}\label{geometric}

We begin with a formal definition and discussion of our geometric condition.

\begin{defn}\label{def-alpha}
Let $\frac{1}{2}\leq\alpha\leq1$. A simply connected domain $U$ satisfies an $\alpha$\emph{-harmonic measure condition} if  there exists a function $C(r)>0$, $r>0$, such that, for every $\zeta\in\partial U$ and $r>0$, we have that
\begin{equation}\label{eqtn:alpha harmonic measure}
\omega(z,\partial V\cap U,V)\leq C(r)|z-\zeta|^\alpha,\;\text{ for }z\in U\cap D(\zeta,r),
\end{equation}
where~$V$ is the component of $U\cap D(\zeta,r)$ that contains $z$.
\end{defn}
\begin{rem*}Observe that the left-hand side of \eqref{eqtn:alpha harmonic measure} can be interpreted as the probability of a Brownian path in~$U$ starting at the point~$z$ first exiting~$V$ through $\partial V\cap U$, which tends to be larger or smaller according to the relative size of $\partial V\cap U$ in $\partial V$. As shown by our Example~\ref{ex:power}, the largest case for $\omega(z,\partial V\cap U,V)$ occurs when the boundary has an external cusp at $\zeta$.
\end{rem*}

The following geometric criteria for satisfying the  $\alpha$-harmonic measure condition are well known and we give only a brief indication of where proofs can be found.

\begin{lem}\label{lem:types of geometry} Assuming the notation of Definition~\ref{def-alpha}:
\begin{enumerate}
\item[(a)] for any simply connected domain~$U$ and $\zeta\in\partial U$, we have
\[
\omega(z,\partial V\cap U,V)\leq C(r)|z-\zeta|^{\frac{1}{2}}, \;\text{ for }z\in V,
\]
so we can take $\alpha=\tfrac12$;
\item[(b)] if every $\zeta\in\partial U$ is the vertex of a sector of angle~$\beta$, $0<\beta\leq\pi$, and side-length~$s>0$, contained in $\C\setminus U$, then
$$
\omega(z,\partial V\cap U,V)\leq C(r)|z-\zeta|^{\frac{\pi}{2\pi-\beta}}, \;\text{ for }z\in V,
$$
where $C(r)$ also depends on $\beta$ and $s$, so we can take $\alpha=\pi/(2\pi-\beta)$;
\item[(c)] if every $\zeta\in\partial U$ lies on the boundary of some disc of radius~$s>0$ contained in $\C\setminus U$, then
$$
\omega(z,\partial V\cap U,V)\leq C(r)|z-\zeta|, \;\text{ for }z\in V,
$$
where $C(r)$ also depends on $s$, so we can take $\alpha=1$.
\end{enumerate}
\end{lem}
Property~(a) can be proved by applying the Milloux--Schmidt inequality \cite[page 289]{HaymanSF2}  to the subharmonic function
\[
u(z):=
\begin{cases}
\omega(z,\partial V\cap U,V), \quad z\in V,\\
0,\quad z\in D(\zeta,r)\setminus V;
\end{cases}
\]
we state a version of the Milloux--Schmidt inequality here, since we will use it again in Section~6.
\begin{lem}[Milloux--Schmidt inequality]\label{MS}
Suppose that~$u$ is subharmonic and continuous in $\{z:|z|\le r\}$ and that
\begin{equation}\label{min0}
\min_{|z|=\rho}u(z)=0, \;\;\text{for } 0<\rho<r.
\end{equation}
Then
\[
u(z)\le \frac{4}{\pi}\max_{|z|=r}u(z)\tan^{-1}(|z|/r)^{1/2},\;\;\text{for } 0<|z|<r.
\]
\end{lem}
Lemma~\ref{MS} shows that in Lemma~\ref{lem:types of geometry} part~(a) we can take $C(r)=4/(\pi r^{1/2})$.

Parts~(b) and~(c) of Lemma~\ref{lem:types of geometry} can be proved by applying standard techniques such as conformal mappings and the maximum principle or Ahlfors' distortion theorem. In fact, using Ahlfors' distortion theorem the values of $\alpha$ stated in~(b) and~(c) can be obtained with slightly smaller exterior domains than sectors and discs, respectively, expressed in terms of a so-called Dini condition; more details can be found in \cite[Section~3.4]{Pommerenke} and \cite[Chapter 3-3, p. 41]{AhlforsCI}.

\begin{rem*}
Using results from \cite{BT}, it is possible to construct individual wandering domains which satisfy (sharply)  any   of the geometries described in  Lemma~\ref{lem:types of geometry}.
\end{rem*}

\subsection{Proof of Theorem~\ref{thm:general result}}\label{proof-thm4.1}

We now prove our main result.  The first part of the argument has its origins in the proof of \cite[Theorem~1.1]{RS-boundaries} but the very general context here, with maps not being surjective, introduces significant new challenges; see also \cite[Lemma~4.1]{Dave-John} for a related generalisation of \cite[Theorem~1.1]{RS-boundaries}.

\begin{proof}[Proof of Theorem \ref{thm:general result}]
We claim that we need only prove the result in the case that $U=\D$ and $z_0=0$. Indeed, if we know that the result holds for a sequence of the form $\hat F_n=F_n\circ \phi_U:\D \to U_n$, where $\phi_U$ is a Riemann map from $\D$ onto $U$ with $\phi_U(0)=z_0$, then the result holds for $F_n:U\to U_n$ since almost every $\xi\in\partial \D$ has the properties that the radial limit $\zeta:=\phi_U(\xi)$ exists and the full radial extension $F_n(\zeta)$ exists (by hypothesis and the invariance of harmonic measure under a conformal map), so
\[
F_n(\zeta)-F_n(z_0)=\hat F_n(\xi)-\hat F_n(0) \to 0\;\text{ as } n\to \infty,
\]
for almost every $\xi\in\partial \D$ and hence for almost every $\zeta\in \partial U$, with respect to harmonic measure.

So, from now on we assume that $U=\D$ and $z_0=0$. We have to show that for almost all $\zeta\in\partial \D$ we have $F_n(\zeta)-F_n(0)\to 0$ as $ n\to \infty$. For each $n\in\N$, let $\zeta_n$ be a point in $\partial U_n$ that satisfies
$$
|F_n(0) - \zeta_n|= \operatorname{dist}(F_n(0), \partial U_n).
$$
By hypothesis, $F_n(0) - \zeta_n\to 0$ as $n\to\infty$.

For every $r>0$, we define the sequence of  open sets $S_n(r)= U_n \cap D(\zeta_n, r)$, for $n\in\N$. We then define the set $E(r)$ of points $\zeta \in \partial \D$ such that $F_n(\zeta)$ eventually lies in $\ov{S_n(r)}$, that is,
\[
E(r)= \bigcup_{N\geq 0} \bigcap_{n\geq N} \left(F_n^{-1}\left(\ov{S_n(r)}\right) \cap \partial \D\right).
\]
Notice that  by the definition of $S_n(r)$ we have
\begin{equation}\label{eq:intersection}
|F_n(\zeta)-\zeta_n|\ra0 \text{ as $n\ra\infty$,\quad  for } \zeta\in \bigcap_{r>0} E(r),
\end{equation}
and, since the sets $E(r)$ decrease as $r$ decreases, for any  sequence of radii $r_k\to 0$ as $k\to 0$ we have $\bigcap_{r>0} E(r)=\bigcap_{k\ge 1} E(r_k)$.  To prove our result it is sufficient to show that the latter set has full harmonic measure in $\partial \D$.

First note that $\ov{U_n}\setminus \ov{S_n(r)}\subset \ov{U_n\setminus S_n(r)}$, so
\begin{equation}\label{eqtn:N}
\partial \D \setminus E(r)\subset\bigcap_{N\geq 0} \bigcup_{n\geq N} (F_n^{-1}(\ov{U_n\setminus S_n(r)})\cap \partial \D).
\end{equation}
We shall use \eqref{eqtn:N} to show that
\begin{equation}\label{eq:zero  measure N}
\omega(0,\partial \D \setminus E(r), \D)=0, \quad\text{for } r>0.
\end{equation}

To prove this, we take $N(r)\in \N$ so large that $F_n(0)\in S_n(r)$ for all $n\ge N(r)$, and then for all $n\ge N(r)$ we define the following sets, shown in Figure~\ref{fig:Pom-anal}:
\begin{itemize}
\item
$V_n(r)$ is the component of $S_n(r)= U_n \cap D(\zeta_n, r)$ that contains $F_n(0)$,
\item
$W_n(r)$ is the component of $F_n^{-1}(V_n(r))\cap \D$ that contains $0$, and
\item
$\Gamma_n(r) := \partial V_n(r)\cap U_n$.
\end{itemize}

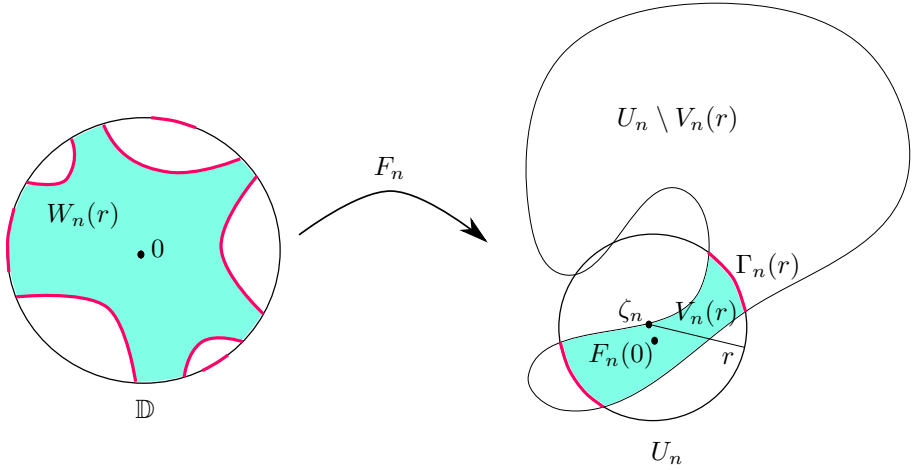
\begin{figure}[hbt!]
\begin{center}
\def\svgwidth{0.8\textwidth}
\begingroup%
  \makeatletter%
    \setlength{\unitlength}{\svgwidth}%
  \makeatother%
  \begin{picture}(1,0.49655153)%
     \put(0,0){\includegraphics[width=\unitlength]{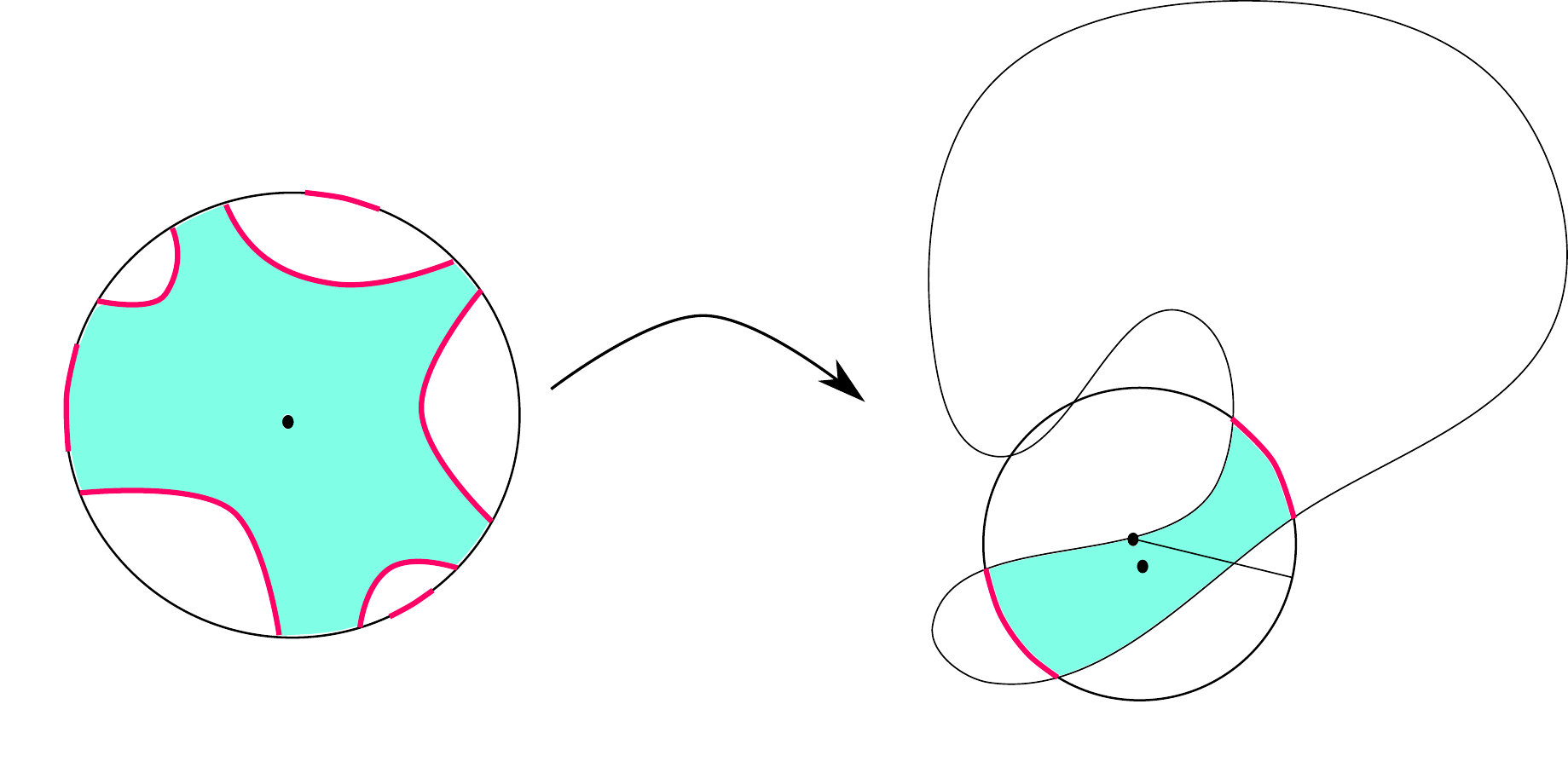}}%
    \put(0.178,0.052){\color[rgb]{0,0,0}\makebox(0,0)[lb]{\small{$\D$}}}%
    \put(0.72487243,0.00324223){\color[rgb]{0,0,0}\makebox(0,0)[lb]{\small{$U_n$}}}%
    \put(0.19451444,0.22518553){\color[rgb]{0,0,0}\makebox(0,0)[lb]{\small{$0$}}}%
    \put(0.08354279,0.25545052){\color[rgb]{0,0,0}\makebox(0,0)[lb]{\small{$W_n(r)$}}}%
    \put(0.65857768,0.10556674){\color[rgb]{0,0,0}\makebox(0,0)[lb]{\small{$F_n(0)$}}}%
    \put(0.8,0.11){\color[rgb]{0,0,0}\makebox(0,0)[lb]{\small{$r$}}}%
    \put(0.68884268,0.35777503){\color[rgb]{0,0,0}\makebox(0,0)[lb]{\small{$U_n\setminus V_n(r)$}}}%
    \put(0.75,0.15){\color[rgb]{0,0,0}\makebox(0,0)[lb]{\small{$V_n(r)$}}}%
    \put(0.815,0.2){\color[rgb]{0,0,0}\makebox(0,0)[lb]{\small{$\Gamma_n(r)$}}}%
    \put(0.69,0.155){\color[rgb]{0,0,0}\makebox(0,0)[lb]{\small{$\zeta_n$}}}%
    \put(0.43080615,0.30759923){\color[rgb]{0,0,0}\makebox(0,0)[lb]{\small{$F_n$}}}%
  \end{picture}%
\endgroup
\end{center}
\caption{\label{fig:Pom-anal} \small \label{fig:main} Illustration  of the proof of Theorem \ref{thm:general result}. The set $V_n(r)$ and  the connected component $W_n(r)$ of its preimage  which contains $0$ are shaded (seagreen online). The arcs $\Gamma_n(r)$ and some of their preimages are drawn thicker (pink online). Notice that there may also be preimages of $\Gamma_n(r)$  in $\partial\D$, since iterates of points in $\partial \D$ may fall in the interior of $U_n$.   Points in  $\D\setminus W_n(r)$, which are shown in white, may be mapped under $F_n$ to  either $U_n\setminus V_n(r)$ or $V_n(r)$, since the preimage of $V_n(r)$ under $F_n$ may have more than one connected component.   Finally, recall that the points $(\zeta_n)_{n\in\N}$  do not form an orbit.}
\end{figure}

Note that $V_n(r)$ and $W_n(r)$ are both simply connected domains, by the maximum principle, and $\Gamma_n(r)$ consists of arcs of $\partial D(\zeta_n,r)$. Moreover, $W_n(r)$ is a Jordan domain since the preimage of $\Gamma_n(r)$ in $\D$ consists of at most countably many curves that accumulate at no point of~$\D$, and the ends of each such preimage curve consist of single points of $\partial \D$.

The latter property follows from a theorem of G.~MacLane \cite[Theorem~1]{MacLane}, applied here to the function $z\mapsto F_n(z)-\zeta_n$, which has finite radial limits almost everywhere on $\partial \D$. MacLane's theorem states that the level curves of any function holomorphic in~$\D$ accumulate at singleton points of $\partial \D$ (rather than at arcs of $\partial \D$) if and only if the function has asymptotic values at all points of a dense subset of $\partial \D$, which is certainly the case for the function $z\mapsto F_n(z)-\zeta_n$.

In particular, $W_n(r)$ has the unique accessibility property and so the restricted map $F_n:W_n(r)\to V_n(r)$ has full radial extension;  see Proposition~\ref{simplecases}.

We now apply Theorem~\ref{Cor:expansion} to the map $F_n:W_n(r)\to V_n(r)$ to deduce that
\begin{equation}\label{lowner1}
\omega(0, F_n^{-1}(\Gamma_n(r))\cap\partial W_n(r),W_n(r))\le \omega(F_n(0), \Gamma_n(r),V_n(r)).
\end{equation}
Next we deduce from the fact that $\overline{V_n(r)}\subset \overline{S_n(r)}$ that, for all $z\in \D$,
\begin{equation}\label{subset}
\omega(z, F_n^{-1}(\overline{U_n\setminus S_n(r)})\cap \partial \D, \D) \le \omega(z, F_n^{-1}(\overline{U_n\setminus V_n(r)})\cap \partial \D, \D).
\end{equation}
Using the interpretation of harmonic measure as the exit distribution of Brownian motion, it is clear that, for all $z\in W_n(r)$, we have
\begin{equation}\label{max-pr}
\omega(z, F_n^{-1}(\overline{U_n\setminus V_n(r)})\cap \partial \D, \D)
 \le \omega(z, F_n^{-1}(\Gamma_n(r))\cap\partial W_n(r),W_n(r)),
\end{equation}
since any Brownian path in $\D$ originating at 0 can only exit $\D$ at a point of $F_n^{-1}(\overline{U_n\setminus V_n(r)})\cap\partial \D$ by first exiting $W_n(r)$ at a point of $F_n^{-1}(\Gamma_n(r))\cap\partial W_n(r)$.

Alternatively, we can apply the maximum principle in the Jordan domain $W_n(r)$ to justify \eqref{max-pr}, by arguing that this inequality holds everywhere on $\partial W_n(r)$, apart from at most countably many points. Indeed, apart from the endpoints of the preimage curves of $\Gamma_n(r)$, these two bounded positive harmonic functions in $W_n(r)$ have the same radial boundary extensions everywhere on $\partial W_n(r)\cap \partial \D$ and the positive harmonic function on the right takes the value~1 everywhere on $F_n^{-1}(\Gamma_n(r))\cap\partial W_n(r) =\partial W_n(r)\cap \D$.

Using the inequalities (\ref{eqtn:N}), \eqref{subset}, \eqref{max-pr} and \eqref{lowner1}, in this order, we deduce that
\begin{eqnarray}\label{eq:main thm}
\omega(0,\partial \D \setminus E(r),\D)
&\leq& \inf_{N\geq N(r)} \sum_{n\geq N} \omega(0,F_n^{-1}\left(\overline{U_n\setminus S_n(r)}\right)\cap \partial \D,\D) \nonumber \\
&\leq& \inf_{N\geq N(r)} \sum_{n\geq N} \omega(0,F_n^{-1}\left(\overline{U_n\setminus V_n(r)}\right)\cap \partial \D,\D) \nonumber \\
&\leq & \inf_{N\geq N(r)} \sum_{n\geq N}\omega(0, F_n^{-1}(\Gamma_n(r))\cap\partial W_n(r),W_n(r))\notag\\
&\leq & \inf_{N\geq N(r)} \sum_{n\geq N} \omega(F_n(0),\Gamma_n(r), V_n(r)).
\end{eqnarray}

Since the domains $U_n$ are all assumed to satisfy the $\alpha$-harmonic measure condition, we deduce from (\ref{eqtn:alpha harmonic measure}) that
$$
\omega(F_n(0), \Gamma_n(r), V_n(r)) \leq C(r)|F_n(0)-\zeta_n|^{\alpha},
$$
for some positive function $C(r), r>0$.
Hence, by \eqref{eq:general1}, the right-hand side of \eqref{eq:main thm} is equal to~0 for every~$r>0$ and \eqref{eq:zero  measure N} follows.

Since a countable union of sets of zero measure has zero measure, we deduce from \eqref{eq:zero  measure N} that for any  sequence of radii $r_k\to 0$ as $k\to \infty$
\[
 \omega(0, \bigcup_{k\ge 1}(\partial \D \setminus E(r_k)), \D)=0.
\]
Therefore, by (\ref{eq:intersection}), we have $|F_n(\zeta)-\zeta_n|\ra0$ as $n \to \infty$ for $\zeta$ in  a subset of $\partial \D$ of full harmonic measure, namely, $\bigcap_{k\ge 1} E(r_k)$.

Since $|\zeta_n-F_n(0)|\ra 0$ by assumption, this  implies that
\[
|F_n(\zeta)-F_n(0)|\leq |F_n(\zeta)-\zeta_n|+|\zeta_n-F_n(0)|\ra0\;\text{ as } n\to \infty,
\]
for $\zeta$ in a set of full harmonic measure, as claimed.
\end{proof}

\begin{rem*}
It can be seen from the proof of Theorem~\ref{thm:general result} that if the domains $U_n$ satisfy an $\alpha$-harmonic measure condition, $1/2 \le \alpha \le 1$, but the orbit $(F_n(0))$ actually approaches parts of the boundaries of the $U_n$ that satisfy a $\beta$-harmonic measure condition, where $\beta>\alpha$, then the convergence condition \eqref{eq:general1} can be weakened to replace the exponent $\alpha$ by $\beta$.
\end{rem*}

\subsection{Example of self-maps of a cardioid domain}\label{cardioid}
We give an example which has the property that in Theorem~\ref{thm:general result} the exponent in condition \eqref{eq:general1} cannot be taken to be greater than~$\frac12$. We consider holomorphic self-maps of the domain
\[
U=\phi(\D),\quad\text{where } \phi(z)=(z-1)^2,
\]
which is the conformal image of~$\D$ under $\phi$, meeting the real axis in the interval $\phi((-1,1))=(0,4)$, with $\partial U$ a cardioid having its inward pointing cusp at $\phi(1)=0$.

By part~(a) of  Lemma~\ref{lem:types of geometry} the domain~$U$ satisfies an $\alpha$-harmonic measure condition for $\alpha=\frac12$. On the other hand,~$U$ does not satisfy such a condition for any $\alpha >\frac12$. Indeed, we can check that the harmonic measure $\omega(z, \partial V\cap U, V)$, where $V=U\cap D(0,r), r>0$, satisfies
\[
\omega(x, \partial V\cap U, V)\sim C(r)x^{1/2}\;\ \text{ as }x \to 0+,
\]
for some constant $C(r)>0$, by using $\phi^{-1}$ to map this harmonic measure to a boundary neighbourhood of $1$ in~$\D$ within which the corresponding harmonic measure behaves near $1$ like a multiple of the distance to $\partial \D$.

\begin{ex}\label{ex:power}
There exists a sequence of holomorphic maps $F_n: U \to U$ such that
\begin{itemize}
\item[(a)] $F_n(1) \to 0 \in \partial U$ as $n \to \infty$ (note that $1=\phi(0)\in U$);
\item[(b)] $\sum\operatorname{dist}(F_n(1),\partial U)^{\alpha}
\begin{cases}
< \infty,\quad \text{for } \alpha>1/2,\\
= \infty,\quad \text{for } \alpha=1/2;
\end{cases} $

\item[(c)] $\{\zeta \in \partial U: F_n(\zeta) \to 0\;\text{as}\;n\to \infty\}$ has measure 0; that is, the Denjoy--Wolff set has measure zero.
\end{itemize}
\end{ex}

\begin{proof}Let $B_n$ be the  sequence of Blaschke products defined as $B_n(z)= M_n(z^{2^n})$, with $M_n(z)= \frac{z+1-1/{n}}{1+(1-1/n)z}$. We then have that
\begin{equation}\label{eq:sharp1}
B_n(0)=1 - \frac{1}{n}\ra1\;  \text{ as } n\ra\infty,
\end{equation}
and
\begin{equation}\label{eq:sharp2}
B_n(\zeta) \nrightarrow 1\;\;\text{as } n \to \infty,\quad \text{for almost every } \zeta \in \partial \mathbb{D},
\end{equation}
where (\ref{eq:sharp2}) follows by applying Example~\ref{ex:z2} with $a_n=1-1/n, n\in\N$, according to which  the orbit of almost every point $\zeta \in \partial \mathbb{D}$  is dense in $\partial\D$.

Now define the self-maps $F_n:U\ra U$ as
\[ F_n= \phi \circ B_n \circ \phi^{-1}.\]
Then, by \eqref{eq:sharp1} and the fact that $\phi(0)=1$, we have
\begin{equation}\label{eq:to1/4}
F_n(1)= \phi (B_n(0))= \phi (1-1/n)=1/n^2 \to 0\;\text{ as } n\to\infty,
\end{equation}
which gives property~(a) and also property~(b), since $\operatorname{dist}(F_n(1),\partial U)= 1/n^2$ for $n\in\N$. Finally, property~(c) holds since $F_n(\zeta) \nrightarrow 0$ as $n\to \infty$ for almost all $\zeta \in \partial U$, by \eqref{eq:sharp2}.
\end{proof}
Note that we can write the maps $F_n:U \to U$ as a forward composition $F_n = f_n \circ \cdots \circ f_1$, where $f_n=\phi \circ b_n \circ \phi^{-1}$ and $b_n$ are the maps given in Example~\ref{ex:z2} such that $B_n=b_n \circ \cdots \circ b_1$.

\section{Empty Denjoy--Wolff sets -- the proof of Theorem D}\label{sect:no DW orbits}
In this section we prove Theorem D. In the proof we need the following basic estimate for the mapping properties of M\"obius transformations, which is also used in Section \ref{sect:examples}.
\begin{lem}[Preimages under M\"obius maps]\label{lem:preimages under Moebius maps}
Consider a M\"obius map of the form
$$M_a(z)=\frac{z+a}{1+az}, \ 0<a<1,\ a=1-\epsilon,$$
and let $T\subset \partial\D$ be an arc $(e^{i\theta},e^{i\phi} )$, with $0<\theta<\phi<2\pi$.
Then there exists $c(\theta,\phi)>0$ such that
\begin{equation}
|M_a^{-1}(T)|\sim c(\theta,\phi) \epsilon\ \  \text{ as $\epsilon\ra 0$}.
\end{equation}
\end{lem}
\begin{proof}
The proof is a direct computation.  For $z,w\in \partial\D$ we have that
$$
M_a^{-1}(z)-M_a^{-1}(w)=\frac{z-a}{1-az}-\frac{w-a}{1-aw}=\frac{(1-a^2)(z-w)}{(1-az)(1-aw)},
$$
so
\begin{align*}
|M_a^{-1}(e^{i\phi})-M_a^{-1}(e^{i\theta})| &=\frac{(1-a^2)|e^{i\phi}-e^{i\theta}|}{|(1-ae^{i\phi})(1-ae^{i\theta})|}\\
&=\frac{\epsilon(2-\epsilon)|e^{i\phi}-e^{i\theta}|}{(|1-e^{i\phi}|+O(\epsilon))(|1-e^{i\theta}|+O(\epsilon))}\\
&=\left(\frac{2|e^{i\phi}-e^{i\theta}|}{|(1-e^{i\phi})(1-e^{i\theta})|}\right) \epsilon+O(\epsilon^2) \;\text{ as } \epsilon\ra 0.
\end{align*}
Since $|M_a^{-1}(T)|\sim |M_a^{-1}(e^{i\phi})-M_a^{-1}(e^{i\theta})|$ as $\epsilon\ra0$, the claim follows.
\end{proof}

\begin{proof}[Proof of Theorem D]

Recall that
\[
F_n(z)=\frac{\lambda_n z+a_n}{1+\lambda_n a_n z},\;\text{ for } n\in\N,
\]
where $(a_n)$ is an increasing sequence in $[0,1)$, with $a_0=0$ and $a_n\to 1$ as $n \to \infty$, and $(\lambda_n)$ is a sequence in $\partial \D$ to be chosen. It is clear that $F_n(0)=a_n$ for $n\in \N$, whatever choice is made of $\lambda_n$, so
\[
F_n(z) \to 1\;\text{ as } n\to \infty, \text{ for all } z\in \D,
\]
by Theorem~A.

We now put $\varepsilon_n=1-a_n$ for $n\in\N$, and assume that
\begin{equation}\label{sum-inf}
\sum_{n=1}^{\infty}(1-a_n)= \sum_{n=1}^{\infty}\varepsilon_n =\infty.
\end{equation}
Then we write
\[
F_n(z)=B_n(\lambda_n z),\quad\text{where}\quad B_n(z)=\frac{z+a_n}{1+a_n z},\;\text{ for } n\in\N.
\]
Let $S=(e^{-i\theta},e^{i\theta})$ be any arc, where $0<\theta <\pi/4$. We show how to choose $(\lambda_n)$ so that no $\zeta\in\partial \D$ has an orbit that eventually lies in~$S$. We have that $B_n^{-1}(S^c)$ is an arc in $\partial \D$ with centre $-1$ and it follows by Lemma~\ref{lem:preimages under Moebius maps} that
\begin{equation}\label{eq:sim-epsn}
|B_n^{-1}(S^c)| \sim c(S)\varepsilon_n\;\text{ as }n\to \infty.
\end{equation}

Therefore, for any choice of $(\lambda_n)$, the preimage $F_n^{-1}(S^c)$ is an arc with centre $-\lambda_n^{-1}$ and with $|F_n^{-1}(S^c)|\sim c(S)\varepsilon_n$ as $n\to\infty$. We introduce intervals $I_n=[\theta_n,\theta_{n+1}], n=0,1,\dots$, such that $\theta_0=0$ and
\[
\theta_{n+1}-\theta_n=|B_n^{-1}(S^c)|,\quad \text{so} \quad \bigcup_{n=0}^{\infty}I_n=[0,\infty),
\]
by \eqref{sum-inf} and \eqref{eq:sim-epsn}.

Now, we choose $\lambda_n \in \partial \mathbb{D}$ such that
\[
-\lambda_n^{-1}=e^{i(\theta_n+\theta_{n+1})/2}.
\]
With this choice we know that if $\theta_n \leq \theta \leq \theta_{n+1}$, then
\[e^{i\theta} \in F_n^{-1}(S^c),\quad\text{so}\quad \lambda_ne^{i\theta} \in B_n^{-1}(S^c)\]
and hence $B_n(\lambda_n e^{i\theta}) \in S^c$. But any $\zeta \in \partial \mathbb{D}$  (including $\zeta=1$) is of the form $\zeta=e^{i\phi_n}$, where $\phi_n \in I_n$ for infinitely many~$n$, so
\[F_n(\zeta)= B_n(\lambda_n e^{i\phi_n}) \in S^c,\quad \text{for such } n.\]
Hence with this choice of $\lambda_n$ we have $F_n(\zeta) \nrightarrow 1$ as $n \to \infty$, as required.
\end{proof}

\begin{rem*}
A small modification to the example in Theorem~D gives a sequence of degree~$p$ Blaschke products, where $p\ge 2$, of the form
\[
F_n(z)= \frac{\lambda_n z^p+a_n}{1+\lambda_n a_n z^p},\quad n\in \N,
\]
such that $F_n(z) \to 1$ as $n\to\infty$ for all $z\in\D$, but the Denjoy--Wolff set of $(F_n)$ is empty. However, note that, as opposed to the M\"obius example of Theorem D,  the sequence $(F_n)$ cannot be written as a forward composition sequence $F_n=f_n\circ \cdots \circ f_1$ where the $f_n$ are self-maps of the unit disc.
\end{rem*}


\section{Proof of Theorem E}

In this section we prove Theorem E. The first part of the proof involves similar techniques to that of Theorem~\ref{thm:general result}, so we omit some of the details here.

First, we can assume that $\zeta_0\in \partial U_0$ is accessible since $L_0$ has positive harmonic measure. We claim we can assume that $U_0=\D$. To see this, let $L'_0=\phi_0^{-1}(L_0)$, where $\phi_0:\D\to U_0$ is a Riemann map. Then $L'_0\subset \partial \D$ has positive measure and the sequence $\hat F_n=F_n\circ \phi_0:\D \to U_n$, $n \in \N$, satisfies $\hat F_n(\xi)-F_n(\zeta_0) \to 0$ as $n\to\infty$ for all $\xi\in L'_0$. Thus if the result holds for $\hat F_n:\D\to U_n$, that is, $\hat F_n(z)-F_n(\zeta_0) \to 0$ as $n\to \infty$ for all $z\in \D$ or equivalently $F_n(\phi_0(z))-F_n(\zeta_0)\to 0$ as $n\to \infty$ for all $z\in \D$, then $F_n(z)-F_n(\zeta_0)\to 0$ as $n\to \infty$ for all $z\in U_0$, as required.

So, from now on we assume that $U=\D$ and $L_0\subset \partial \D$ has positive measure. Let $\zeta_n:=F_n(\zeta_0)$ for $n\in\N$, so by hypothesis $\dist(\zeta_n,\partial U_n)\to 0$ as $n\to \infty$ (possibly, $\dist(\zeta_n,\partial U_n)= 0$ for all $n \in \N$ if $\zeta_n\in\partial U_n$ for all $n\in\N$). Then take any $z_0\in \D$ and put $z_n=F_n(z_0)\in U_n$, for $n\in\N$. To obtain our result it is sufficient, by Theorem~A, to prove that
\begin{equation}\label{zn-zetan}
|z_n-\zeta_n|\to 0\;\text{ as }n\to\infty.
\end{equation}

By  the Severini--Egorov theorem (see \cite[Chapter~3]{Royden} or \cite[Theorem~1.3.26 and Exercise 1.4.31]{tao}, for example), there is a compact set $E_0 \subset L_0 \subset \partial \D$ such that
\begin{equation}\label{E0-L0}
\omega(z_0,E_0,\D) \ge \tfrac12\omega(z_0,L_0,\D)\;\text{ and }\; |F_n(\zeta)- \zeta_n|\to 0 \text{ as } n\to\infty,  \;\text{ uniformly for } \zeta\in E_0.
\end{equation}

Hence there exists a positive sequence $r_n\to 0$ as $n\to\infty$ such that
\begin{equation}\label{zeta-n}
F_n(E_0)\subset D(\zeta_n, r_n)\quad\text{and} \quad D(\zeta_n, r_n)\cap \partial U_n \neq \emptyset ,\;\text{ for } n\in\N.
\end{equation}

We prove \eqref{zn-zetan} by showing that, for some positive absolute constant~$C>1$ we have
\begin{equation}\label{eq:constC}
z_n\in D(\zeta_n,Cr_n), \;\text{ for  }n\in\N.
\end{equation}
It is sufficient to consider only those values of~$n$ for which $z_n\notin D(\zeta_n,r_n)$.

For such~$n$, we let~$V_n$ denote the component of $U_n\setminus \overline{D(\zeta_n,r_n)}$ that contains~$z_n$, $W_n$ denote the component of $F_n^{-1}(V_n)\subset \D$ that contains $z_0$, and put $\Gamma_n=\partial V_n\cap U_n$. Note that $V_n$ and $W_n$ are both simply connected domains, by the maximum principle. Moreover, as in the proof of Theorem~\ref{thm:general result}, $W_n$ is a Jordan domain since the part of  the preimage of $\Gamma_n$ that lies in $\D$ consists of at most countably many curves that accumulate at no point of~$\D$, and the ends of each such preimage curve consist of single points of $\partial \D$.

We can now apply Theorem~\ref{Cor:expansion} to $F_n$ in $W_n$ to deduce that
\begin{equation}\label{lowner2}
\omega(z_0, F_n^{-1}(\Gamma_n)\cap\partial W_n,W_n)\le \omega(z_n, \Gamma_n,V_n).
\end{equation}
Next observe that $E_0$ lies in $\partial \D \setminus \partial W_n$, by the definitions of $L_0$, $V_n$ and $W_n$. Therefore, since any Brownian path in $W_n$ originating at $z\in W_n$ can only exit $\D$ at a point of $E_0$ by first exiting $W_n$ at a point of $\partial W_n\cap \D$, we have
\begin{equation}\label{maxpr1}
\omega(z, E_0, \D)\le \omega(z, F_n^{-1}(\Gamma_n)\cap\partial W_n,W_n),\;\text{ for } z\in W_n,  n\in\N.
\end{equation}
Alternatively, we can prove \eqref{maxpr1} by using the maximum principle, because the harmonic measure on the left has radial boundary extension~0 almost everywhere on $\partial \D \cap \partial W_n$ (since $E_0 \subset\partial \D \setminus \partial W_n$) and the harmonic measure on the right takes the value~1 almost everywhere on $\D\cap\partial W_n$.

Now we estimate the right-hand side of \eqref{lowner2} from above. To do this we choose an unbounded closed connected set~$K_n$ which lies in $\C\setminus U_n$ and contains a point in $D(\zeta_n,r_n)\cap \partial U_n$. This set exists since~$U_n$ is a simply connected proper subset of $\C$. Then let~$\Omega_n$ denote the component of the complement of $K_n \cup \overline{D(\zeta_n,r_n)}$ that contains~$z_n$, which is a simply connected domain such that $V_n\subset \Omega_n$ and $\Gamma_n\subset \partial\Omega_n$. Then, by the maximum principle again, this time applied first in $V_n$ and second in $\Omega_n$, we have
\begin{equation}\label{maxpr2}
\omega(z, \Gamma_n,V_n) \le \omega(z, \Gamma_n,\Omega_n)\le \omega(z, \partial \Omega_n \cap U_n,\Omega_n),\;\text{ for } z\in V_n, n\in\N.
\end{equation}
In particular, \eqref{maxpr2} holds for $z=z_n$.

Now the right-hand function in \eqref{maxpr2} can be estimated from above by first mapping $\Omega_n$ to a bounded domain using $w=\psi_n(z):= 1/(z-\zeta_n)$ and then applying the Milloux-Schmidt inequality to the harmonic function
\[
u(w)= \omega(\psi_n^{-1}(w), \partial \Omega_n \cap U_n,\Omega_n),\quad w\in \psi_n(\Omega_n), n\in\N.
\]
Note that the domain $\psi_n(\Omega_n)$ lies in $D(0,R_n)$, where $R_n:=r_n^{-1}$, and~$u$ vanishes on the part of the boundary of $\psi_n(\Omega_n)$ that lies in $D(0,R_n)$. After extending~$u$ to be~0 throughout $D(0,R_n)\setminus \psi_n(\Omega_n)$ and hence subharmonic in $D(0,R_n)$, we deduce from the Milloux--Schmidt inequality, Lemma~\ref{MS}, that for $n\in\N,$ we have
\[u(w) \le \frac{4}{\pi} \tan^{-1}\left(\frac{|w|}{R_n} \right)^{\frac{1}{2}}\leq \frac{4}{\pi}\left(\frac{|w|}{R_n}\right)^{\frac12}, \;\text{ for } w\in \psi_n(\Omega_n),
\]
and in particular, taking $w=\psi_n(z_n)$, that
\[
\omega(z_n,\partial \Omega_n \cap U_n,\Omega_n)=u(\psi_n(z_n))
\le \frac{4}{\pi}\left(\frac{1}{R_n|z_n-\zeta_n|}\right)^{\frac12}=\frac{4}{\pi}\left(\frac{r_n}{|z_n-\zeta_n|}\right)^{\frac12}.
\]
Combining this estimate with \eqref{E0-L0}, \eqref{lowner2}, \eqref{maxpr1} and \eqref{maxpr2}, we deduce that, for $n\in\N$,
\[
\omega(z_0,L_0, \D)\le 2\omega(z_0,E_0, \D) \le \frac{8}{\pi}\left(\frac{r_n}{|z_n-\zeta_n|}\right)^{\frac12}, \quad\text{so}\quad
|z_n-\zeta_n|\le \frac{64r_n}{(\pi\omega(z_0,L_0, \D))^{2}}.
\]
Since $r_n\to 0$ as $n\to\infty$ and $\omega(z_0,L_0,\D)>0$ by hypothesis, this proves \eqref{eq:constC}, as required.


\section{Classifying forward compositions and proof of Theorem F}\label{sec:thmF}

In this section we focus on forward compositions of holomorphic maps between simply connected domains. We first classify such sequences in terms of hyperbolic distances, extending the classification of wandering domains given in \cite{BEFRS}. We then prove a general result concerning contracting sequences which implies Theorem F. At the end of the section we give an example which shows that the conclusion of Theorem F does not hold if the hypothesis of contracting is omitted.

\subsection{Hyperbolic classification}
In \cite[Theorem~A]{BEFRS}, we classified simply connected wandering domains into three distinct types: \textit{contracting}, \textit{semi-contracting} or \textit{eventually isometric}, according to whether hyperbolic distances between orbits of points, for almost all points, tend to 0, decrease but do not tend to 0, or are eventually constant, respectively. In Section~\ref{intro} of this paper we extended the definition of `contracting' to the more general context of sequences of holomorphic maps $F_n:U\to U_{n}$ between simply connected domains; see Definition~\ref{cont}.

In this section, we extend the concepts of `semi-contracting' and `eventually isometric' from wandering domains to sequences of holomorphic maps that can be expressed as forward compositions of holomorphic maps between simply connected domains; that is, sequences of the form $F_n=f_n \circ \cdots \circ f_1$, where $f_n:U_{n-1}\to U_n$ are holomorphic maps.

We also give a criterion to discriminate between the three types of forward compositions based on the concept of \emph{hyperbolic distortion} \cite[Sect. 5,11]{BeardonMinda}.

\begin{defn}[Hyperbolic distortion]\label{HDdef}
If $f:U\ra V$ is a holomorphic map between two hyperbolic domains $U$ and $V$, then the \emph{hyperbolic distortion} of $f$ at $z$ is
\[
\|Df(z)\|_U^V:=\lim_{z'\ra z}\frac{\dist_V(f(z'),f(z))}{\dist_U(z',z)}=\frac{\rho_V(f(z)) |f'(z)|}{\rho_U(z)},
\]
where $\rho_U(z)$ denotes the hyperbolic density at $z \in U$.
\end{defn}

Note that in the following result the holomorphic maps need not be surjections.

\begin{thm}[Hyperbolic classification theorem] \label{thm:classification}
Let $U_n, n \ge 0$ be a sequence of simply connected domains (not necessarily distinct or disjoint) and $f_n:U_{n-1} \to U_n, n\in \N$ be a sequence of holomorphic maps. Define the forward compositions $F_n=f_n\circ \cdots \circ f_1, n\in \N,$ and the set
\[
E=\{(z,z')\in U_0\times U_0 : F_n(z)=F_n(z') \text{\ for some $n\in\N$}\}.
\]
For $n\in \N$, let $\lambda_n(z)$ denote the hyperbolic distortion $\|Df_n(F_{n-1}(z))\|_{U_{n-1}}^{U_{n}}$.

Then exactly one of the following holds:
\begin{itemize}
\item[\rm(1)] $\dist_{U_n}(F_n(z), F_n(z'))\lran c(z,z')= 0 $ for all $z,z'\in U_0$, that is, the sequence $(F_n)$ is contracting on $U_0$; this case occurs if and only if $\sum_{n=1}^{\infty} (1-\lambda_n(z))=\infty$.
\item [\rm(2)] $\dist_{U_n}(F_n(z), F_n(z'))\lran c(z,z') >0$ and $\dist_{U_n}(F_n(z), F_n(z')) \neq c(z,z')$
for all $(z,z')\in (U_0 \times U_0) \setminus E$, $n \in \N$, in which case we say that the sequence $(F_n)$ is {\em  semi-contracting} on $U_0$.
\item[\rm(3)] There exists $N>0$ such that for all $n\geq N$, $\dist_{U_n}(F_n(z), F_n(z')) = c(z,z') >0$ for all $(z,z') \in (U \times U) \setminus E$, in which case we say that the sequence $(F_n)$ is {\em eventually isometric} on $U_0$; this case  occurs if and only if $\lambda_n(z) = 1$, for $n$ sufficiently large.
\end{itemize}
\end{thm}

By the Schwarz-Pick Lemma, the eventually isometric case occurs if and only if $f_n:U_{n-1}\to U_{n}$ is univalent for all sufficiently large~$n$.

The proof of Theorem~\ref{thm:classification} is essentially identical to that of the corresponding result for simply connected wandering domains, \cite[Theorem~A]{BEFRS}, depending as it does on using Riemann maps to reduce the situation to classifying non-autonomous systems of forward compositions of self-maps of the unit disc that fix the point~0, as discussed in \cite[Section~2]{BEFRS}.

The hyperbolic classification theorem does not hold in general for sequences of holomorphic maps
between simply connected domains that are not defined by forward composition, as the following example shows.

\begin{ex}\label{Mercer} Consider the sequence $(F_n)$ of degree~2 Blaschke products defined as follows. For $z\in\D$ and $n\in\N$,
\[
F_{n}(z):=z\phi_{2/n}(\phi_{1/2}(z)),\quad\text{where}\quad \phi_a(z):=\frac{z-a}{1-\overline az},\;\;a\in\D.
\]
Then, $\dist_{\D}(F_n(1/2),F_n(0))\to 0 \text{ as } n\to\infty,$ but
\begin{equation}\label{nonzero}
\dist_{\D}(F_n(z),F_n(z'))\nrightarrow 0\;\text{ as } n\to\infty,\; \text{ for all } z, z' \in \D, \;(z,z') \neq (0,1/2).
\end{equation}
\end{ex}
\begin{proof}
For $n\in\N$, $F_{n}(0)=0, F_{n}(1/2)=-1/n$, so $\dist_{\D}(F_n(1/2),F_n(0))\to 0 \text{ as } n\to\infty,$ and
\[
F_n(z)=z\frac{\phi_{1/2}(z)-2/n}{1-(2/n)\phi_{1/2}(z)}\to z\phi_{1/2}(z)\;\text{ as }n\to \infty,\;\; \text{for all }z\in\D,
\]
from which the property \eqref{nonzero} readily follows.
\end{proof}
The form of the functions in Example~\ref{Mercer} was suggested by work of P.~Mercer (\cite{mercer}).

\subsection{Proof of Theorem F}
Theorem~F is a special case of  part~(b) of the following result, obtained by including the assumption that the interior orbits converge to the boundary.

\begin{thm}\label{thm:Pomgen}
Let $F_n: U_0 \to U_n$, $n\in\N$, be a sequence of holomorphic maps between simply connected domains, each having a full radial extension, with the additional hypothesis that
\begin{equation}\label{fn-bdry-prop7}
F_n(\zeta) \in \partial U_n,\;\text{ for almost all } \zeta \in \partial U_0.
\end{equation}
Also, suppose that $(F_n)$ is contracting in $U_0$.
\begin{itemize}
\item[(a)]
If there are measurable subsets $L\subset \partial U_0$ and $L_n\subset \partial U_n$, $n\in\N$, such that $L= F_n^{-1}(L_n)$, for $n\in\N$, up to a set of harmonic measure~0, then $L$ has either full or zero harmonic measure with respect to $U_0$.
\item[(b)]
If in addition $F_n= f_n \circ \cdots \circ f_1,$ where $f_n :U_{n-1} \to U_n$, $n\in\N$, are holomorphic maps between simply connected domains, each with a full radial extension, then for any $\zeta_0\in \partial U_0$ the set
\begin{equation} \label{eq:clarifying}
\{\zeta\in\partial U_{0} : |F_n(\zeta)-F_n(\zeta_0)|\ra0\}
\end{equation}
either has full or zero harmonic measure with respect to~$U_0$.
\end{itemize}
\end{thm}

For the proof of Theorem \ref{thm:Pomgen} we need the  following `strong mixing' result due to  Pommerenke \cite[Theorem~1]{Pommerenke-ergodic}.

\begin{lem}[Pommerenke]\label{thm:pommerenke}
Let $(G_n)_{n\geq1}$ be a sequence of inner functions such that
\[
G_n(0)=0\quad\text{and}\quad G_n(z) \to 0 \; \text{ as } n \to \infty,\; \text{ for all } z\in \D,
\]
and suppose there are measurable subsets $E$ and $E_n$, $n\in\N$, of $\partial \D$ such that $E= G_n^{-1}(E_n)$, for $n\in\N$, up to a set of measure~0. Then, for all arcs $A\subset \partial \D$, we have
\begin{equation}\label{mixing}
|A\cap E| = \frac{1}{2\pi}|A|.|E|.
\end{equation}
It follows that the set~$E$ has either full or zero Lebesgue measure.
\end{lem}

\begin{rem*}
The statement of Lemma~\ref{thm:pommerenke} is different from that of \cite[Theorem~1]{Pommerenke-ergodic} in several respects. First, for simplicity, we have normalised the inner functions to fix 0. Second, the proof of the identity \eqref{mixing} is obtained by following the proof of \cite[Theorem~1]{Pommerenke-ergodic} while exchanging the roles of the sets $E_n$ and $E$ in that proof. Third, the final statement of Lemma~\ref{thm:pommerenke} follows by applying \eqref{mixing} to a nested sequence of arcs converging to a point of density of~$E$, when such a point exists.
\end{rem*}

\begin{proof}[Proof of Theorem \ref{thm:Pomgen}]
Let $z_0$ be an arbitrary point in $U_0$ and let $h_n$ be a univalent map from $U_n$ onto $U_0$ with $h_n(F_n(z_0))=z_0$. Then the map  $h_n \circ F_n$ is  a self-map of $U_0$ fixing $z_0$ (see Figure \ref{fig:Pom-anal}). Now let $\phi: \mathbb{D} \to U_0$ be a Riemann map such that $\phi(0)=z_0$. Accessible boundary points of $U_0$ correspond to points of $\partial \mathbb{D}$ where $\phi$ has radial limits (see Section~\ref{sect:boundary maps}).

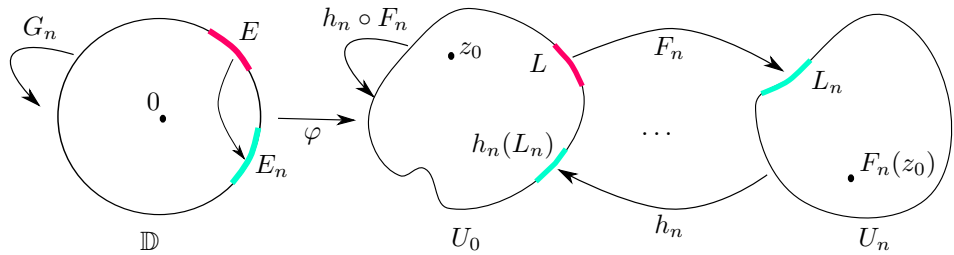
\begin{figure}[hbt!]
\begin{center}
\def\svgwidth{0.8\textwidth}
\begingroup%
  \makeatletter%
    \setlength{\unitlength}{\svgwidth}%
  \makeatother%
   \begin{picture}(1,0.2603515)%
    \put(0,0){\includegraphics[width=\unitlength]{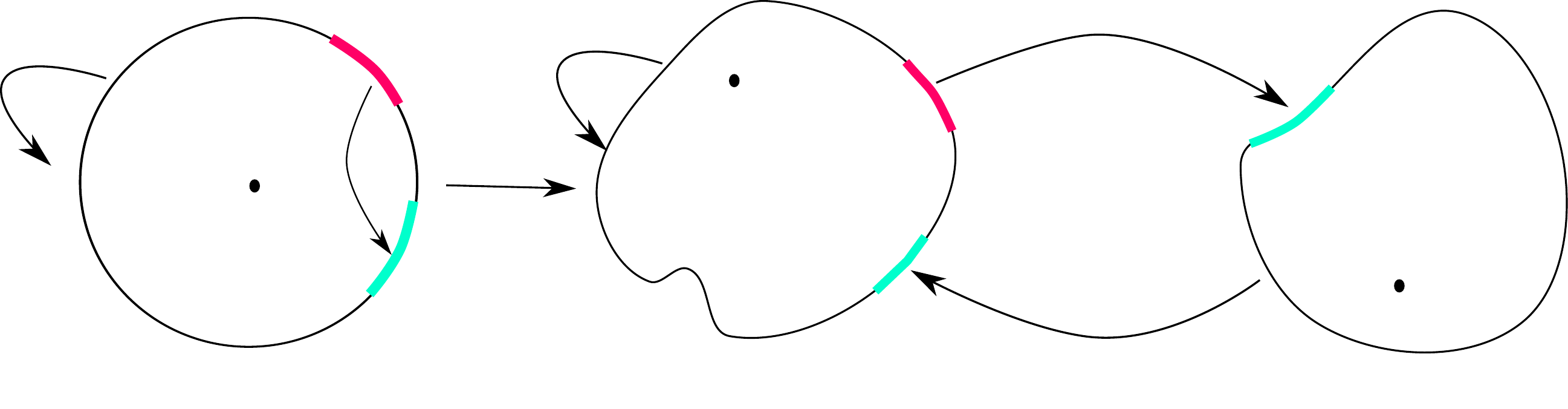}}
    \put(0.13850199,0.0014){\color[rgb]{0,0,0}\makebox(0,0)[lb]{\small{$\D$}}}%
    \put(0.24274871,0.22556075){\color[rgb]{0,0,0}\makebox(0,0)[lb]{\small{$E$}}}%
    \put(0.26,0.08){\color[rgb]{0,0,0}\makebox(0,0)[lb]{\small{$E_n$}}}%
    \put(0.01399325,0.22663144){\color[rgb]{0,0,0}\makebox(0,0)[lb]{\small{$G_n$}}}%
    \put(0.85037972,0.16978653){\color[rgb]{0,0,0}\makebox(0,0)[lb]{\small{$L_n$}}}%
    \put(0.55121414,0.19226582){\color[rgb]{0,0,0}\makebox(0,0)[lb]{\small{$L$}}}%
    \put(0.68302459,0.20565798){\color[rgb]{0,0,0}\makebox(0,0)[lb]{\small{$F_n$}}}%
    \put(0.68404638,0.01549675){\color[rgb]{0,0,0}\makebox(0,0)[lb]{\small{$h_n$}}}%
    \put(0.33153002,0.23722443){\color[rgb]{0,0,0}\makebox(0,0)[lb]{\small{$h_n\circ F_n$}}}%
    \put(0.46844944,0.0014){\color[rgb]{0,0,0}\makebox(0,0)[lb]{\small{$U_0$}}}%
    \put(0.90066512,0.0014){\color[rgb]{0,0,0}\makebox(0,0)[lb]{\small{$U_n$}}}%
    \put(0.49,0.09928325){\color[rgb]{0,0,0}\makebox(0,0)[lb]{\small{$h_n(L_n)$}}}%
    \put(0.3121161,0.115){\color[rgb]{0,0,0}\makebox(0,0)[lb]{\small{$\phi$}}}%
    \put(0.66054529,0.125){\color[rgb]{0,0,0}\makebox(0,0)[lb]{\small{     $\ldots$}}}%
      \put(0.14536942,0.15154541){\color[rgb]{0,0,0}\makebox(0,0)[lb]{\small{$0$}}}%
    \put(0.47632162,0.20727276){\color[rgb]{0,0,0}\makebox(0,0)[lb]{\small{$z_0$}}}%
    \put(0.90166914,0.07875869){\color[rgb]{0,0,0}\makebox(0,0)[lb]{\small{$F_n(z_0)$}}}%
  \end{picture}%
\endgroup%
\end{center}
\caption{\label{fig:Pom-anal} \small Illustration of the proof of Theorem \ref{thm:Pomgen}. The sets $E_n$ and $L_n$ are represented as arcs for simplicity but are in fact just measurable.}
\end{figure}

For $n\geq1$ we define the map $G_n:\D\ra\D$ as
\[
G_n=\phi^{-1}\circ h_n\circ F_n \circ \phi,
\]
and note that  $G_n(0)=0$ for $n\in\N$. Since the sequence $(F_n)$ is contracting we have that $\dist_{U_n}(F_n(z),F_n(z_0))\ra 0$ as $n\to\infty$ for all $z \in U_0$. Also, since $\phi^{-1}\circ h_n: U_n\ra\D$ is a conformal map, we have, for every $w\in\D$,
\begin{align*}
\dist_\D(G_n(w), G_n(0))&=\dist_\D(\phi^{-1}\circ h_n\circ F_n\circ\phi(w), \phi^{-1}\circ h_n\circ F_n(z_0)) \\
&=\dist_{U_n}(F_n(\phi(w)), F_n(z_0))\ra0 \;\text{ as }n\to \infty,
\end{align*}
which means that $(G_n)$ is also contracting.

Next define the sets $E_n\subset \partial \D$  as
\begin{equation}\label{eq:Pom3}
E=\phi^{-1}(L)\quad\text{and} \quad E_n=\phi^{-1}(h_n(L_n)),\;\text{ for } n\in\N.
\end{equation}
By  \eqref{fn-bdry-prop7} and \eqref{eq:Pom3} we have, for $n\in\N$, up to a set of measure zero (by an argument similar to the one used in the proof of Proposition~\ref{composing-radial extns}), that
\begin{align*}
G_n^{-1}(E_n)&=\phi^{-1}\circ F_n^{-1}\circ h_n^{-1}\circ\phi(E_n)\\
&=\phi^{-1}\circ F_n^{-1}\circ h_n^{-1}\circ h_n(L_n)\\
&=\phi^{-1}(L)=E.
\end{align*}
Suppose now that~$L$ has positive harmonic measure in $U_0$, so $E$ has positive Lebesgue measure in $\partial \D$. Since $F_n$ maps $\partial U_0$ to $\partial U_n$,  $G_n$ is an inner function for every~$n$.
We can now apply Lemma \ref{thm:pommerenke} to the sequence $(G_n)$ and deduce that $|E|=2\pi$, and so $L \subset \partial U_0$ has full harmonic measure.

For part~(b), let~$L$ be the set in \eqref{eq:clarifying}, and  assume that for some $\zeta_0\in\partial U_0$ the  set $L$ has positive harmonic measure in $U_0$. Then define the set $L_n \subset \partial U_n$, for $n\in\N$, as
\begin{equation}\label{eq:Pom1}
L_n:= \{\zeta \in \partial U_n : |f_{n+m} \circ \cdots \circ f_{n+1}(\zeta)-F_{m+n}(\zeta_0)| \to 0 \;\text{ as }m \to \infty\}.
\end{equation}

By definition,
\begin{equation} \label{eq:Pom2}
 L= \{\zeta\in\partial U_{0} : |F_n(\zeta)-F_n(\zeta_0)|\ra0\} = F_n^{-1}(L_n),\;\text{ for } n\in\N,
\end{equation}
up to a set of harmonic measure~0, so we can apply part~(a) of the theorem to deduce that~$L$ has full harmonic measure in~$U_0$.
\end{proof}

\subsection{An example related to Theorem~F}
The following example shows that Theorem~F (and Theorem~\ref{thm:Pomgen}) fail without the hypothesis which requires the sequence of maps to be contracting. The example consists of a sequence of  M\"obius self-maps $M_n$ of $\D$, and hence it can also be viewed as a forward composition $M_n=m_n \circ \cdots \circ m_1$ of M\"obius self-maps by defining $m_n=M_n\circ M^{-1}_{n-1}$.

\begin{ex}  \label{ex:isometric}
There exists a sequence $(M_n)$ of M\"obius self-maps of $\D$ such that $M_n(1)=1$ for all $n$, and
\begin{enumerate}[\rm (a)]
\item $M_n(z)\to 1$ as $n\to\infty$, for all $z\in\D$;
\item $M_n(\zeta)\to 1$ as $n\to\infty$, for all $\zeta=e^{i\theta}$, $|\theta|<\pi/2$; but
\item $M_n(\zeta) \not\to 1$ as $n\to\infty$, for all $\zeta=e^{i\theta}$, $\pi/2 \leq |\theta| \leq \pi$.
\end{enumerate}
\end{ex}
\begin{proof}
We define the maps first on the right half-plane. For points $\zeta\in\partial \mathbb{H}$ and $a>1$ we shall consider affine self-maps of $\mathbb{H}$ of the form
\[
\widehat{M}(z; \zeta,a)=a(z-\zeta)+\zeta.
\]
Clearly $\widehat{M}(\infty)=\infty$,  $\widehat{M}(\zeta)=\zeta$.

For $j=1,2,\ldots$ we divide the segment $[-i,i]$ in the imaginary axis into $j$ segments $I_{j,k}$,  $k=0,\ldots,j-1$, each of length $2/j$, that is
\[
I_{j,k}=\left[-1+\frac{2k}{j},-1+\frac{2(k+1)}{j}\right]i,\quad k=0,\ldots,j-1,
\]
with midpoint
$
\zeta_{j,k}=\left(-1+\frac{2k+1}{j}\right)i.
$
Hence, for all $k=0,\ldots,j-1$,
\[
\left| I_{j,k}\right|=\frac{2}{j}
\text{ \ \ and \ \  }  \bigcup_{k=0}^{j-1} I_{j,k}=[-i,i].
\]
Now, for all $k=0,\ldots,j-1$,  we define the maps
\[
\widehat{M}_{j,k}(z)=\widehat{M}(z; \zeta_{j,k},j)= j (z-\zeta_{j,k}) + \zeta_{j,k},\quad z\in \overline{\mathbb{H}}.
\]
Then, on the one hand, for $\zeta\in I_{j,k}$ we have
\begin{equation}\label{eq:notinfty}
|\widehat{M}_{j,k}(\zeta)| \leq j |\zeta-\zeta_{j,k}|+ 1 \leq j(1/j)+1=2,
\end{equation}
and on the other hand, for all $|z|>1$ and since $|\zeta_{j,k}|<1$,
\begin{equation}\label{eq:infty}
|\widehat{M}_{j,k}(z)| \geq j |z-\zeta_{j,k}| -1 \geq j(|z|-1)-1,\quad 0\le k\le j-1, j=1,2\ldots.
\end{equation}
Now we arrange the maps $\widehat{M}_{j,k}$ into a single sequence $\widehat{M}_n=\widehat{M}_{j,k}$, where $n=\frac12j(j+1)+k$, $0\le k\le j-1$, $j=1,2\ldots$, and observe that  for all $z\in \overline{{\mathbb H}}$ with $|z|>1$ we have that $\widehat{M}_n(z)\to \infty$ as $n\to\infty$, by \eqref{eq:infty}. On the other hand, every $z\in [-i,i]$ belongs to an infinite number of segments $I_{j,k}$, so there exists an infinite number of distinct values of $n$ for which $|\widehat{M}_n(z)|<2$ by (\ref{eq:notinfty})and therefore $\widehat{M}_n(z)\not\to \infty$ as $n\to\infty$.

Finally, we define the corresponding maps on the closed unit disc by $M_n(z)=\beta^{-1} \circ \widehat{M}_n \circ\beta(z)$, where $\beta:\D \to \mathbb{H}$ is the M\"obius map such that $\beta(1)=\infty$, $\beta(-1)=0$ and $\beta(\pm i)=\pm i$. By the discussion above, properties (b) and (c) are immediate. Property~(a) follows immediately from Theorem~A, for example.
\end{proof}

\section{Examples related to the ADM dichotomy}\label{sect:examples}
In this section we give three examples related to Theorem~\ref{thm:DM}, part~(b). Recall that this states that for an inner function $f: \D\to \D$ with a Denjoy--Wolff point and such that $\sum_{n\in\N}(1-|f^n(z)|)=\infty$, for some $z\in\D$, almost all boundary orbits of $f$ are dense in $\partial \D$. Here we show that this result fails to hold in general if iterates of inner functions are replaced by forward compositions of inner functions.

In Section~\ref{ADMb-fails}, we give two examples which show this failure: the first consists of forward compositions of M\"obius self-maps of $\D$ and is therefore an isometric sequence in the sense of Theorem~\ref{thm:classification}, and the second consists of forward compositions of Blaschke products of degree~2, which turns out to be a semi-contracting sequence. Then, in Section~\ref{ADMb-holds} we give an example of a forward composition of distinct Blaschke products of degree~2 which is contracting and for which the conclusion of Theorem~\ref{thm:DM}, part~(b) does hold -- the techniques used in the proof of this example have some independent interest.

These examples raise the question of whether part~(b) of the ADM dichotomy holds for {\em contracting} forward compositions of inner functions (or indeed for more general contracting sequences), which we discuss at the end of this section.

\subsection{Forward compositions for which Theorem~\ref{thm:DM}, part~(b) fails}\label{ADMb-fails}
Our first example is a sequence of M\"obius self-maps of $\D$.
\begin{ex}\label{Mob-dvgt}
Let $(a_n)$ be an increasing sequence on $[0,1)$ such that $a_0=0$ and $\lim_{n \to \infty}a_n=1$, let $m_n$, $n\in\N$, be the M\"obius map such that $m_n(\pm 1)=\pm 1$ and $m_n(a_{n-1})=a_n$, and define $M_n=m_n\circ\cdots\circ m_1$, that is,
\[M_n(z)=\frac{z+a_n}{1+a_nz}, \quad \text{for } n \geq 1.\]
Then, every  point $\zeta \in \partial \mathbb{D}\setminus\{-1\}$ satisfies that $\lim_{n\to\infty} M_n(z)=1$, regardless of whether $\sum_{n\geq 0} (1-a_n)$ converges or not.
\end{ex}
\begin{proof}
We consider the cross-ratio
\[\frac{(z_1-z_3)(z_4-z_2)}{(z_3-z_2)(z_1-z_4)},\]
which is invariant under M\"obius transformations. Take $z_1=1,z_2=-1,z_3=0,z_4= \zeta \in \partial \mathbb{D} \setminus \{\pm1\}.$ Then the invariance under $M_n$ gives
\[0 \neq \frac{(1-0)(\zeta+1)}{(0+1)(1-\zeta)}= \frac{(1-a_n)(M_n(\zeta)+1)}{(a_n+1)(1-M_n(\zeta))}.\]
Since $a_n \to 1$ we have $\frac{1-a_n}{a_n+1}\to 0$ as $n \to \infty$, so
\[\frac{1-M_n(\zeta)}{M_n(\zeta)+1} \to 0 \quad \text{as}\; n \to \infty,\]
as required.
\end{proof}

As announced,  a similar effect to that in Example~\ref{Mob-dvgt} can be achieved with a  forward composition of Blaschke products $(b_n)$ of degree two. In this example, $B_n=b_n\circ \cdots \circ b_1 \to 1$ as $n\to \infty$ in $\D$ and the summability condition of Theorem~\ref{thm:DM}, part~(b) fails, and yet $B_n \to 1$ as $n\to \infty$ on a whole arc of $\partial \D$ with centre~1.
\begin{ex}\label{Bn-to-1}
There exists a sequence of degree two Blaschke products $(b_n)$ of the form
\begin{equation}
\label{bn-def}
b_n(z)= \left(\frac{z+\mu_n}{1+\mu_nz}\right)^2,
\end{equation}
where $(\mu_n)$ is a decreasing sequence such that $\mu_n \to 1/3$ as $n \to \infty$, such that
\begin{enumerate}[(a)]
\item  $B_n:=b_n\circ\cdots\circ b_1$ satisfies $ \sum_{n=1}^{\infty}(1-B_n(0))=\infty,$
\item  $\lim_{n \to \infty} B_n(z)=1$, for $z \in \D \cup \{\zeta\in\partial\D:|\zeta-1|<\rho\},$ for some $\rho>0.$
\end{enumerate}
\end{ex}
\begin{proof}
First note that the map $z \mapsto \left(\frac{z+\mu}{1+\mu z}\right)^2$ has an attracting fixed point at~1 for $\mu \in (1/3,1)$ and a parabolic fixed point at 1 for $\mu=1/3$. It is convenient to construct the sequence $(\mu_n)$ by making a change  of variables to the right half-plane $\mathbb{H}=\{z:\operatorname{Re}z>0\}$ via the map $ \alpha(z)= \frac{1+z}{1-z}$.

Now $z \mapsto \frac{z+\mu_n}{1+\mu_n z}$ is conjugate to $z\mapsto \lambda_n z$ under $ \alpha$, where $\lambda_n= \frac{1+\mu_n}{1-\mu_n}$ and we require $(\lambda_n)$ to be a decreasing sequence tending to 2 as $n \to \infty$. Also $z \mapsto z^2$ is conjugate to $z \mapsto \frac{1}{2}(z+1/z)$ under $\alpha$. We deduce that $b_n$ in ${\mathbb D}$ is conjugate to the map $z \mapsto \frac{1}{2}(\lambda_nz+\frac{1}{\lambda_nz})$ in ${\mathbb H}$. Abusing notation, we shall denote this map also by $b_n$, and also write $B_n=b_n\circ \cdots \circ b_1$ in ${\mathbb H}$. The sum in statement~(a) now becomes
\[
\sum_{n=1}^{\infty}\left|1- \alpha^{-1}(B_n(1))\right|=\sum_{n=1}^{\infty} \frac{2}{B_n(1)+1},
\]
which diverges if and only if $\sum_{n= 1}^{\infty} 1/B_n(1)$ diverges.

We choose $(\lambda_n)$ in such a way that $B_n(1)=n+1$, for $n\in\N$, so that this divergence is assured. To achieve this we need
\begin{equation}
\label{eq:b_lamda}
B_n(1)= b_n(n)=\frac{1}{2}\left(\lambda_n n+ \frac{1}{\lambda_n n}\right)=n+1,
\end{equation}
so we define
\begin{equation}
\label{eq:lamda_n}
\lambda_n := \frac{n+1+ \sqrt{n^2+2n}}{n} = 1+\frac1n+\left(1+\frac2n\right)^{1/2} = 2\left(1+ \frac{1}{n}\right)+ O(1/n^2)\;\text{ as } n \to \infty.
\end{equation}

We have $B_n(z) \to \infty$ for all $z$ in $\mathbb{H}$, by Theorem~A, since this property holds for $z=1$. So to prove~(b) we need to show that, for all sufficiently large~$|y|$:
\begin{equation}
\label{eq:B_n}
B_n(iy) \to \infty \;\text{ as } n \to\infty.
\end{equation}
Now
\begin{equation}
\label{b-beta}
b_n(iy)= \frac{1}{2}\left(\lambda_n iy+ \frac{1}{\lambda_n iy}\right)=\frac{1}{2}\left(\lambda_n iy- \frac{i}{\lambda_n y}\right)=: i\beta_n(y),
\end{equation}
say. By symmetry we need only to consider $y>0$. It is easy to check that the maps $y\mapsto\beta_n(y)$ have a unique positive fixed point $y^*_n =\frac12 \sqrt{n} + o(1)$ and, if $y>y_n^*$ then $\beta_n(y)> y$. Hence we choose to check that if $y_n> n^\frac34$,  for some $n\ge N$, say, then $y_{n+1}=\beta_n(y_n) > (n+1)^\frac34$ also.

We have
\begin{align*}
y_{n+1} &=\beta_n(y_n)= \frac{1}{2}\left(\lambda_n n^{3/4}-  \frac{1}{\lambda_n n^{3/4}} \right) \nonumber \\
&= \frac{1}{2} \left( (2(1+1/n)+O(1/n^2))n^{3/4}- \frac{1}{(2(1+1/n)+O(1/n^2))n^{3/4}}\right) \nonumber \\
&= \frac{n+1}{n^{1/4}}+O(1/n^{5/4})-\frac{n^{1/4}}{4(n+1+O(1/n))}\;\text{ as } n\to \infty. \nonumber
\end{align*}
We claim that for $n$ large enough, $y_{n+1}>(n+1)^{3/4}.$ This holds since
\[\frac{n+1}{n^{1/4}}-(n+1)^{3/4}> \frac{n^{1/4}}{4(n+1+O(1/n))}+O(1/n^{5/4}),\]
that is,
\[ \left(\frac{n+1}{n}\right)^{1/4}-1>\frac{n^{1/4}}{4(n+1)^{3/4}(n+1+O(1/n))}+ O\left(1/n^2\right),\]
which follows from
\[\frac{1}{4n} > \frac{n^{1/4}}{4(n+1)^{7/4}+O(1/n^{1/4})} + O\left(\frac{1}{n^2}\right),\]
true for~$n$ large enough. Hence there exists $N \in \mathbb{N}$ such that $y_N>N^{3/4}$ implies that $y_n>n^{3/4}$, for $n \geq N$. We deduce from this and from \eqref{b-beta} that, for all $y >N^{3/4}$,
\[
b_n\circ\cdots \circ b_N(iy) \to \infty \;\text{ as } n\to \infty
\]
and so \eqref{eq:B_n} holds if we relabel the $b_n$ starting with $b_N$.
\end{proof}
\begin{rem*}
In Example~\ref{Bn-to-1} it is natural to ask how large is the subset of $\partial \D$ of points such that $B_n \to 1$ as $ n\to \infty$. It is plausible that this set is dense in $\partial\D$. However, the sequence of Blaschke products in Example~\ref{Bn-to-1} is semi-contracting, as we now show, so we cannot use Theorem~\ref{thm:Pomgen} to deduce that $\lim_{n \to \infty} B_n(\zeta)=1$ for almost all $\zeta \in \partial\D$. Indeed, if we use Definition~\ref{HDdef} and \eqref{eq:lamda_n} to compute
\begin{align*}
\|\operatorname{D}b_n(n)\|_{\mathbb H}^{\mathbb H}&=\frac{\rho_{{\mathbb H}}(n+1)}{\rho_{{\mathbb H}}(n)}\,b'_n(n) =\frac{1/(n+1)}{1/n}\,\frac12\left(\lambda_n-\frac{1}{\lambda_n n^2}\right)\\
&=\frac{n}{n+1}\left(1+1/n+O(1/n^2)\right)\\
&=1+O(1/n^2)\;\text{ as } n\to \infty,
\end{align*}
we see that
\[
\sum_{n=1}^{\infty}\left(1-\|\operatorname{D}b_n(a_n)\|_{\mathbb H}^{\mathbb H} \right)<\infty,
\]
which by Theorem \ref{thm:classification} implies that the sequence $(B_n)$ is semi-contracting.
\end{rem*}

\subsection{Forward compositions for which Theorem~\ref{thm:DM}, part (b) holds}\label{ADMb-holds}
In contrast to the two previous examples, where we saw that the second part of the ADM dichotomy does not hold in general, we now show that this second part does hold for certain types of forward compositions that form contracting sequences. Note that the following example is the basis of Example~\ref{ex:power}.

\begin{ex}\label{ex:z2}
Let $(a_n)$ be any increasing sequence on $[0,1)$ such that $a_0=0$ and $\lim_{n \to \infty}a_n=1$, and for $n\geq 0$ let
$
M_n(z)=\frac{z+a_n}{1+a_nz}.
$
For $n\in\N$ we define the Blaschke products
\[
b_n(z):=M_n\left((M_{n-1}^{-1}(z))^2\right),
\]
and
\begin{equation}\label{Gn-Mn}
B_n(z):=b_n\circ \cdots \circ b_1 (z)
= M_n(z^{2^n})=\frac{z^{2^n}+a_n}{1+a_nz^{2^n}}.
\end{equation}
Then, $B_n(z)\to 1$ for all $z\in \D$. If, in addition,
\begin{equation}\label{dvgt}
\sum_{n\geq 0}(1-B_n(0))=\sum_{n\geq 0}(1-a_n) =\infty,
\end{equation}
then the orbit of almost every point of $\partial\D$ under $(B_n)$ is dense in $\partial \D$.
\end{ex}

We need some preliminaries before we can prove Example \ref{ex:z2}. We first state a `weak independence' version of the second Borel--Cantelli lemma which (in the setting of forward compositions, rather than iterates) acts as a substitute for techniques from the proof of the Poincar\'e recurrence theorem, used by Doering and Ma\~n\'e to prove Theorem~\ref{thm:DM}, part~(b).

\begin{lem}[Borel--Cantelli]\label{thm:Borel-Cantelli}
If $E_n$ are measurable subsets of $[0,1]$ such that $\sum |E_n|= \infty$ and
\[
|E_m \cap E_n|\leq C |E_m|.|E_n|, \;\text{ for }n>m\geq L,
\]
where $C \geq 1$ and $L \in \mathbb{N}$, then
\[
|\limsup_{n \to \infty}E_n|=|\{x\in [0,1]: x \in E_n \text{ \rm{infinitely often}}\}|= \left|\bigcap_{N \geq 0} \bigcup_{n \geq N}E_n\right| >0.
\]
\end{lem}
Lemma~\ref{thm:Borel-Cantelli} is due to Ciesielski and Taylor \cite{Ciesielski} and independently to Lamperti \cite{Lamperti}, whose proof is admirably short. Another version is due to Petrov \cite{Petrov}, which has the same hypotheses but the stronger conclusion that $|\limsup_{n \to \infty}E_n|\ge 1/C$. Closely related results of this type are due Kochen and Stone \cite{Kochen-Stone}, and Yan \cite{Yan}.

We use Lemma~\ref{thm:Borel-Cantelli} to prove a `shrinking target' result. This description of a result that concerns the size of the set of points where a dynamical system visits infinitely often a sequence of balls whose size shrinks to~0 with~$n$ seems to originate with Hill and Velani \cite{HV} in the context of metric Diophantine approximation.
\begin{lem}[Shrinking target]\label{lem:shrinking}
Let $g(z)=z^2$ and let $(\epsilon_n)$ be any positive decreasing sequence in $(0,1]$ such that $\sum_{n=1}^{\infty}\epsilon_n=\infty$. Further assume that the arcs $I_n\subset \{e^{i\theta}: |\theta+\pi|\le \tfrac12\epsilon_n\}$ satisfy $|I_n| \ge c\epsilon_n$ for $n\in\N$, where $c>0$ is independent of~$n$. Then
\begin{equation}\label{g-inf}
|\{\zeta\in \partial \D: g^n(\zeta) \in I_n \text{ \rm{infinitely often}}\}|>0.
\end{equation}
\end{lem}
Lemma~\ref{lem:shrinking} can also be deduced from a more general and more precise result due to Philipp~\cite[Theorem~2A]{Philipp}, which is based on another version of the second Borel--Cantelli lemma~\cite[Theorem~3]{Philipp}. Since the details of the proof in~\cite{Philipp} are involved and we do not need the extra precision here, we include a proof of Lemma~\ref{lem:shrinking} based solely on Lemma~\ref{thm:Borel-Cantelli}.
\begin{proof}[Proof of Lemma~\ref{lem:shrinking}]
First, $g^{-n}(\{e^{i\theta}: |\theta+\pi|\le \tfrac12\epsilon_n\})$ consists of $2^n$ arcs, each of length $\epsilon_n/2^n$ and centred at the $2^n$-th roots of $-1$. Let $A_n$ denote the union of these $2^n$ arcs; see Figure~\ref{fig:thmz2}.

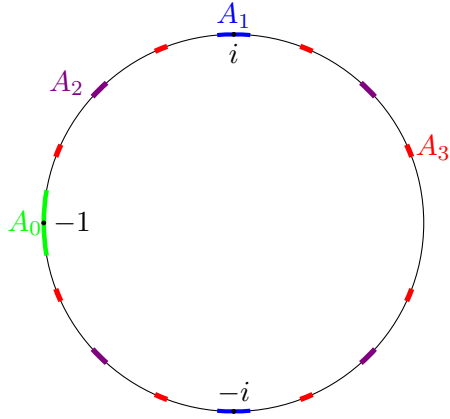
\begin{figure}[hbt!]
\centering
\begin{tikzpicture}[scale=2.5] 
\draw (0,0) circle [radius=1];
\draw[line width=0.6mm, green] (cos{170},sin{170}) arc(170:190:1);
\draw[line width=0.55mm, blue] (cos{85},sin{85}) arc(85:95:1);
\draw[line width=0.55mm, blue] (cos{265},sin{265}) arc(265:275:1);
\draw[line width=0.7mm, violet] (cos{42},sin{42}) arc(42:48:1);
\draw[line width=0.7mm, violet] (cos{132},sin{132}) arc(132:138:1);
\draw[line width=0.7mm, violet] (cos{222},sin{222}) arc(222:228:1);
\draw[line width=0.7mm, violet] (cos{312},sin{312}) arc(312:318:1);
\draw[line width=0.7mm, red] (cos{20.5},sin{20.5}) arc(20.5:24.5:1);
\draw[line width=0.7mm, red] (cos{65.5},sin{65.5}) arc(65.5:69.5:1);
\draw[line width=0.7mm, red] (cos{110.5},sin{110.5}) arc(110.5:114.5:1);
\draw[line width=0.7mm, red] (cos{155.5},sin{155.5}) arc(155.5:159.5:1);
\draw[line width=0.7mm, red] (cos{200.5},sin{200.5}) arc(200.5:204.5:1);
\draw[line width=0.7mm, red] (cos{245.5},sin{245.5}) arc(245.5:249.5:1);
\draw[line width=0.7mm, red] (cos{290.5},sin{290.5}) arc(290.5:294.5:1);
\draw[line width=0.7mm, red] (cos{335.5},sin{335.5}) arc(335.5:339.5:1);
\draw[fill] (-1,0) circle [radius=0.01];
\draw[fill] (0,1) circle [radius=0.01];
\draw[fill] (0,-1) circle [radius=0.01];
\node at (-0.85,0){$-1$};
\node at (0,-0.9){$-i$};
\node at (0,0.9){$i$};
\node at (0,1.1){\color{blue}$A_1$};
\node at (-1.1,0){\color{green}$A_0$};
\node at (-0.88,0.74){\color{violet}$A_2$};
\node at (1.05,0.4){\color{red}$A_3$};
\end{tikzpicture}
\caption{\label{fig:thmz2} The sets $A_n$.}
\end{figure}

Then
\begin{equation}\label{Anprops}
|A_n|= 2^n \times \frac{\epsilon_n}{2^n}= \epsilon_n,\quad n \geq 1.
\end{equation}

We will now find an upper bound for the length $|A_m \cap A_n|$ whenever $n>m\ge 1$. We consider two cases. First, take $n=m+p$, where $p\in\N$ is so small that $\frac{1}{2^{p+1}}\geq \epsilon_m$. The centre points of the arcs of $A_m$ are the $2^m$-th roots of $-1$, which are of the form $\exp(2\pi i \theta_{k,m}),$ where $\theta_{k,m}=k/2^{m+1}$ and $k \in \{1,\ldots, 2^{m+1}-1\}$ is odd.
Therefore the centre points of the arcs of $A_{m+p}$ that are closest to the arc of $A_m$ with centre $\exp(2\pi i \theta_{k,m})$, where $k \in \{1,\ldots, 2^{m+1}-1\}$ is odd, are
\[
\exp(2\pi i((2^pk\pm 1)/2^{m+p+1})=\exp(2\pi i (\theta_{k,m} \pm 1/2^{m+p+1})).
\]
Since $\frac{1}{2^{p+1}}\geq \epsilon_m$ and the sequence $(\epsilon_n)$ is decreasing, we deduce that
\[
\frac{1}{2}.\frac{\varepsilon_m}{2^m}+\frac{1}{2}.\frac{\varepsilon_{m+p}}{2^{m+p}}<\frac{\varepsilon_m}{2^m}\leq \frac{1}{2^{m+p+1}},
\]
so the  arcs of $A_{m+p}$ closest to  $\exp(2\pi i \theta_{k,m})$ do not meet the arc of $A_m$ centred at that point. Hence, for this range of values of~$n$ we have $A_m \cap A_n = \emptyset$.

Now we consider the case where $n=m+p, p\in \N$ and $\frac{1}{2^{p+1}}<\varepsilon_m$. Recall that each of the $2^m$ arcs of $A_m$ has length $\varepsilon_m/2^m$ and the angle between the centres of the arcs $A_{m+p}$ is $2\pi/2^{m+p}$. Hence the number of arcs of $A_{m+p}$ that can meet each arc of $A_m$ is at most
\[
\frac{\varepsilon_m/2^m}{2\pi/2^{m+p}}+1= \frac{2^p \varepsilon_m}{2\pi}+1,
\]
we deduce in this case, and hence in all cases, that
\begin{align}
\label{Am-An}
|A_m \cap A_n| &= |A_m \cap A_{m+p}|\nonumber\\
& \le 2^m\left(\frac{2^p \varepsilon_m}{2\pi}+1\right)\frac{\varepsilon_{m+p}}{2^{m+p}} \nonumber \\
&= \varepsilon_m \varepsilon_{m+p}\left(\frac{1}{2\pi}+\frac{1}{2^p\varepsilon_m}\right) \nonumber \\
&\leq  3|A_m|.|A_{n}|.
\end{align}
We now map the sets $A_m$ by $z\mapsto \arg z/(2\pi)$ and apply Lemma~\ref{thm:Borel-Cantelli} to the resulting subsets, $E_m$ say, of $[0,1]$. Since
\[
|E_m\cap E_n|\le 6\pi |E_m|.|E_n|, \quad\text{for }n>m\ge 1,
\]
we deduce by \eqref{Anprops} and Lemma~\ref{thm:Borel-Cantelli} that
\[
|\limsup_{n \to \infty}E_n|>0,
\]
and hence \eqref{g-inf} holds in the case when $I_n = \{e^{i\theta}: |\theta+\pi|\le \tfrac12\epsilon_n\}$, for $n\in\N$.

In the general case that $I_n \subset \{e^{i\theta}: |\theta+\pi|\le \tfrac12\epsilon_n\}$, for $n\in\N$, let $A'_n=g^{-n}(I_n)$. Then $A'_n$ consists of $2^n$ arcs, each one contained in a different arc of $A_n$. Note that the arcs in $A'_n$ have length at least $c|A_n|= c\,\epsilon_n/2^n$, so $|A'_n|\ge c\epsilon_n$.  Then, by \eqref{Am-An}, we have for all $n\in\N$
\begin{align*}
\label{Am-An}
|A'_m \cap A'_n| &\le |A_m \cap A_n|\\
&\le 3|A_m|.|A_{n}|\\
&\le (3/c^2) |A'_m|.|A'_{n}|.
\end{align*}
Therefore, we can once again apply Lemma~\ref{thm:Borel-Cantelli} to deduce that \eqref{g-inf} holds in general.
\end{proof}
 To complete the proof of Example~\ref{ex:z2} we need the following property of the sequence $(B_n)$ defined in \eqref{Gn-Mn}.
\begin{lem}\label{contracting}
The sequence of maps $B_n:\D \to \D, n\ge1,$ is contracting.
\end{lem}
\begin{proof}
We evaluate the hyperbolic distortion along the orbit of 0 under $B_n$. By~\eqref{Gn-Mn}, for each $n\in\N$:
\begin{eqnarray}
\|\operatorname{D}B_n(0)\|_{\D}^{\D}&=& \lim_{z\to 0} \frac{\operatorname{dist}_{\D}(B_n(z),B_n(0))}{\operatorname{dist}_{\D}(z,0)}\nonumber \\
&=& \lim_{z\to 0} \frac{\operatorname{dist}_{\D}(z^{2^n},0)}{\operatorname{dist}_{\D}(z,0)}\nonumber \\
&=& \lim_{z\to 0} \log \frac{1+|z|^{2^n}}{1-|z|^{2^n}}\Big/\log \frac{1+|z|}{1-|z|}=0.\nonumber
\end{eqnarray}
It follows that
\[\sum_{n=1}^{\infty} (1-\|\operatorname{D}B_n(0)\|_{\mathbb{D}}^{\mathbb{D}})= \infty,\]
and so $(B_n)$ is indeed contracting, by Theorem~\ref{thm:classification}.
\end{proof}

\begin{proof}[Proof of Example~\ref{ex:z2}]
Since $B_n(0)=a_n\to 1$ as $n \to \infty$, it is clear that $B_n(x)\to 1$ for all $x\in (0,1)$, so this convergence holds throughout $\D$ by Theorem~A.

Next we assume that \eqref{dvgt} holds. To complete the proof it is sufficient to show that if~$T$ is any non-trivial closed arc of $\partial \D \setminus \{1\}$, then almost every point of $\partial \D$ has an orbit under $(B_n)$ that visits $T$ infinitely often. To do this we choose an arc $S\subset \partial \D$ with centre~1 so that $T\subset S^c$. It follows from Lemma~\ref{lem:preimages under Moebius maps} that there are constants $c_S>c_T>0$ such that, for all $n \ge 1$ large enough, the preimage $M_n^{-1}(S^c)$ contains an arc of $\partial \D$ centred at $-1$ of length $c_S(1-a_n):=\epsilon_n,$ say, and the preimage $M_n^{-1}(T)$ contains an arc $I_n\subset \{e^{i\theta}: |\pi+\theta|\le \tfrac12 \epsilon_n\}$ of length at least $c_T(1-a_n)=(c_T/c_S)\epsilon_n$. By hypothesis, $\sum_{n=1}^{\infty}\epsilon_n = \infty$, and also $(\epsilon_n)$ is decreasing since $(a_n)$ is increasing.

Now, since $M_n(I_n)\supset T$,
\begin{align*}
\limsup_{n \to \infty} B_n^{-1}(T) &= \{\zeta \in \partial \mathbb{D}:B_n(\zeta) \in T \text{ infinitely often}\}\\
&\supset  \{\zeta \in \partial \mathbb{D}:g^n(\zeta) \in I_n \text{ infinitely often}\},
\end{align*}
so, by Lemma~\ref{lem:shrinking},
\begin{equation}\label{T-positive}
|\limsup_{n \to \infty} B_n^{-1}(T)|>0.
\end{equation}

We now use Lemma~\ref{contracting} and Theorem~\ref{thm:Pomgen}, part~(a) to deduce from \eqref{T-positive} that $\limsup_{n \to \infty} B_n^{-1}(T)$ actually has full measure with respect to $\partial \mathbb{D}$, which implies that $B_n(\zeta)\in T$  infinitely often for almost every $\zeta \in \partial \mathbb{D}$, as required.

To do this we define, for $n \ge 0$,
\begin{equation}\label{eq:Pom1}
L_n:= \limsup_{m \to \infty} (b_{n+m}\circ \cdots \circ b_{n+1})^{-1}(T),
\end{equation}
that is, the set of $\zeta\in \partial \D$ whose orbit under the sequence of maps $b_{n+m}$, $m\ge 1$, visits~$T$ infinitely often; in particular,
\[
L_0:= \limsup_{m \to \infty} B_n^{-1}(T).
\]
By construction we have $L_0=B_n^{-1}(L_n)$, for $n\ge 0$, and so the set $\limsup_{n \to \infty} B_n^{-1}(T)$ has full measure, by \eqref{T-positive}, Lemma~\ref{contracting} and Theorem~\ref{thm:Pomgen}, part~(a), as required.
\end{proof}
\begin{rem*}\label{rem8}
Example~\ref{ex:z2} can be generalised considerably; for example, in Lemma~\ref{lem:shrinking} we can replace the map $g(z)=z^2$ by $g(z)=z^p$ for any integer $p\ge 2$ by applying \cite[Theorem~2A]{Philipp} instead of the direct proof of the lemma given here. And it seems plausible that in Example~\ref{ex:z2} we may be able to replace the function~$g(z)=z^2$ used in the definition of $b_n$ by   any contracting Blaschke product (or even inner function), that fixes~0, by using a recent result \cite[Theorem~5.11]{Rosenblatt} which gives a shrinking target result for inner functions that are ergodic on $\partial \D$, and hence contracting   by \cite[Theorem~3.1]{doering-mane}.
\end{rem*}

\subsection{Part~(b) of the ADM dichotomy} \label{ADMproof}
Earlier in this section we showed that part~(b) of the ADM dichotomy does not hold in general in the setting of forward compositions, at least in the isometric and semi-contracting cases. This leaves open the possibility that it holds in the contracting case.

To provide some evidence that this may be the case (apart from our failure to produce a contracting counter-example!), we outline briefly how to prove Theorem~\ref{thm:DM}, part~(b) by using results of Aaronson, Doering and Ma\~n\'e, and Neuwirth, most of which appear in the text \cite{Aar97}; see also \cite[Section~2]{BFJKboundaries} for a convenient summary of the concepts from ergodic theory mentioned here and the relationships amongst them.

All the steps in the proof concern the behaviour of the inner function $f$ on the boundary of $\D$. (Recall that an inner function is defined at almost all boundary points.)
\begin{itemize}
\item the condition that $\sum_{n\geq 0}(1-|f^n(z)|)=\infty$ for some $z\in\D$ implies that~$f$ is conservative and hence ergodic on $\partial \D$, by \cite[Propositions~6.1.7 and~6.1.8]{Aar97};
\item since $f$ is an inner function it is non-singular with respect to the Lebesgue measure on $\partial \D$, by  \cite[Proposition~6.1.1]{Aar97};
\item
hence~$f$ is non-singular, conservative and ergodic on $\partial \D$, which implies, by \cite[Proposition~1.2.2]{Aar97}, that for any set~$E$ of positive measure in $\partial \D$, the iterates $f^n(\zeta)$ visit~$E$ infinitely often for almost every $\zeta \in \partial \D$, a strong recurrence property which implies the density of orbits property stated in Theorem~\ref{thm:DM}, part~(b).
\end{itemize}

The reader may wish to compare these arguments to those in the proof of \cite[Theorem C]{BFJKboundaries} which uses a slightly different path to arrive at the same conclusion.

The relevance of the proof outlined above to our question about contracting sequences is as follows. We see from the first step in the proof that, for an inner function~$f$, the divergence condition given in Theorem~\ref{thm:DM}, part~(b) implies the ergodic property, and we know from \cite[Theorem~3.1]{doering-mane} that the iterates of an inner function which is ergodic on $\partial \D$ form a contracting sequence in the disc.

So, it is reasonable to ask whether some version of part (b) of the ADM dichotomy holds for {\it any} contracting composition (or indeed sequence) of holomorphic functions between simply connected domains. However, concepts like invariance,  ergodicity, recurrence and density of orbits only make sense when the domains $U_n$ are all identical so that boundary points can be identified. In particular, these concepts do not make sense for wandering domains.

We can, however, ask whether a weaker version of part (b) of the dichotomy holds, with a different type of proof. More precisely, suppose that $F_n:U \to U_n$ is a contracting forward composition (or perhaps sequence) of holomorphic functions between simply connected domains for which interior orbits converge to the boundary sufficiently slowly (in some precise sense that takes the geometry of the domains into account). Then it is plausible that the Denjoy--Wolff set has measure zero.

To prove such a result, possible tools include versions of the {\it second} Borel--Cantelli lemma (such as the one used to prove Example~\ref{ex:z2}) and L\"owner's lemma (the case of equality, mentioned in the remark following Lemma~\ref{thm:Ransford non continuous}), together with Pommerenke's strong boundary mixing result \cite[Theorem~1]{Pommerenke-ergodic} in the case of forward compositions. Such an approach would complement our proof of the generalisation of part~(a) of the ADM dichotomy, where we used the L\"owner's lemma inequality, together with an argument similar to that used to prove the {\it first} Borel--Cantelli lemma. The Borel--Cantelli lemmas have a dichotomy similar to Theorem~\ref{thm:DM}, in which convergent and divergent series lead to sets of zero measure or positive measure, respectively, and in that sense they are a natural tool to attack this problem.

\section{Versions of results with the spherical metric}\label{sec:spherical}
It is natural to ask to what extent our results hold when the Euclidean metric for measuring distances to the boundary of a domain is replaced by the spherical metric. Using the spherical metric has the advantage that sequences tending to $\infty$ are included but the disadvantage that points may be close together in the spherical metric but far apart in the Euclidean metric. As an illustration, consider the sequence
\[
F_n(z)=M_n(z)+n, \quad z\in \D,\; n=1,2,\ldots,
\]
where $(M_n)$ is the sequence of M\"obius self-maps of $\D$ in Example~\ref{ex:isometric}. In this case, the spherical distances between all pairs of points of $\ov{F_n(\D)}$ tends to 0 as $n\to\infty$, but the subtle behaviour of $F_n$ on $\partial \D$ is only seen when using the Euclidean metric.

Also, our hypotheses on the functions $F_n: U\to U_n$ are so general that Theorem~A does not hold when expressed in terms of the spherical metric. For example, sequences of the form
\begin{equation}\label{non-normal1}
F_n(z) = C_nB(z)+z,\quad z\in\D,\; n=1,2,\ldots,
\end{equation}
where $B$ is a Blaschke product, $C_n\to\infty$ as $n\to\infty$, $U=\D$ and $U_n=\{z:|z|<|C_n|+1\}$, and ones of the form
\begin{equation}\label{non-normal2}
F_n(z) = (2z)^{p_n},\quad z\in\D,\; n=1,2,\ldots,
\end{equation}
where $p_n\in \N$ and $p_n\to\infty$ as $n\to \infty$, $U=\D$ and $U_n=\{z:|z|<2^{p_n}\}$, show that we can have $F_n(z)\to \infty$ as $n\to\infty$ for some $z\in \D$ but not all. However, it follows easily from the Euclidean version of Theorem~A that with the extra hypothesis of normality we have the following spherical version; here we denote spherical distance by $\chi$.
\begin{thm}\label{DWthm-spher}
Let $F_n: U \to U_n$ be a sequence of holomorphic maps between simply connected domains and suppose that $(F_n)$ is normal in $U$. If there exists $z_0 \in U$ such that
\[
\chi(F_n(z_0), \partial U_n) \to 0 \;\text{ as } n \to \infty,
\]
then, for all $z \in U$,
\[
\chi(F_n(z),F_n(z_0)) \to 0 \;\text{ as } n \to \infty.
\]
\end{thm}
The role of the normality hypothesis in Theorem~\ref{DWthm-spher} is to ensure that if $\chi(F_{n_k}(z_0),\infty) \to 0$ as $k\to\infty$ for any subsequence $(F_{n_k})$, then $\chi(F_{n_k}(z),\infty) \to 0$ as $k\to\infty$ for all $z\in U$.

Next we state a spherical version of Theorem~C in which the assumption of normality is not required.
\begin{thm}\label{thm:fast-spher}
Let $F_n:U \to U_n$ be a sequence of holomorphic maps between simply connected domains, each with a full radial extension to $\partial U$, and suppose  that there exists $z_0 \in U$ such that
\begin{equation} \label{eq:power-spher}
\sum_{n=0}^{\infty} \operatorname{dist_{\chi}}(F_n(z_0),\partial U_n)^{1/2}<\infty.
\end{equation}
Then for almost all points $\zeta \in  \partial U$ we have
\[
\chi(F_n(\zeta),F_n(z_0))\to 0 \;\;\text{as}\;n \to \infty.
\]
\end{thm}
\begin{rems*}
\begin{itemize}
\item[1.]
Examples of the form \eqref{non-normal1}, with $C_n=n^3$ for instance, show that in the spherical setting the hypothesis \eqref{eq:power-spher} is not independent of $z_0\in U$, in contrast to the independence of condition \eqref{eq:power} in Theorem~C; see the first remark after Theorem~\ref{thm:general result}.
\item[2.]
A special case of Theorem~\ref{thm:fast-spher} can be found in \cite[remark after Theorem~1.1]{RS-boundaries}, where it is observed that if $U$ is a Baker domain of an entire function $f$ such that for some $z_0\in U$ and all sufficiently large $n$ we have
\begin{equation}\label{eq:geom series}
|f^{n+1}(z_0)|\ge k|f^n(z_0)|, \;\text{where } k>1,
\end{equation}
then, for almost all $\zeta\in\partial U$, we have
\[
f^n(\zeta)\to \infty \;\text{ as } n\to\infty.
\]
Indeed, since for a point $z\in\C$ we have that $\dist_{\chi}(z,\infty)\sim 1/|z|$, the assumption (\ref{eq:geom series}) implies that the series in (\ref{eq:power-spher}) is convergent by comparing it to a geometric series. Similarly, by observing for example that the $p$-series $\sum_{n=0}^{\infty} \frac{1}{n^p}<\infty$ for $p>1$, we obtain that the claim of Theorem~\ref{thm:fast-spher} holds whenever for some $z_0\in U$ and all sufficiently large $n$ we have $|f^n(z_0)|\geq n^{2p}$ for $p>1$, or analogously,
\begin{equation}\label{eq:p series}
|f^{n+1}(z_0)|\ge \left(1+\frac{1}{n}\right)^{2p}|f^n(z_0)|, \;\text{where } p>1.
\end{equation}

\end{itemize}
\end{rems*}
There is also a spherical version of Theorem~E in which we again need an extra hypothesis, one that implies normality.
\begin{thm}\label{thmE-spher}
Let $F_n:U \to U_n$ be a sequence of holomorphic maps between simply connected domains, each with a full radial extension, and suppose that there exists $R>0$ such that
\begin{equation}\label{cond:normality}
(\hat{\mathbb C}\setminus U_n) \cap\{z:|z|=R\}\ne \emptyset,\;\text{ for } n\in\N.
\end{equation}
If there exists $\zeta_0\in \partial U$ such that $\chi
(F_n(\zeta_0),\partial U_n) \to 0$ as $n \to \infty$ and
\[
L_0 := \{\zeta \in \partial U : \chi(F_n(\zeta),F_n(\zeta_0)) \to 0\;\; \text{as } n \to \infty\}
\]
has positive harmonic measure with respect to~$U$, then, for all $z\in U$, we have
\[
\chi(F_n(z),F_n(\zeta_0)) \to 0\;\; \text{as } n \to \infty.
\]
\end{thm}
The extra hypothesis \eqref{cond:normality} ensures that each point of $\partial U_n$, $n\in\N$, lies in a continuum that is exterior to $U_n$ and contains a point of $\{z:|z|\le R\}$. This continuum is required because replacing the Euclidean metric by the spherical metric means that the spherical discs needed in the proof sometimes contain $\infty$, and in this circumstance the condition \eqref{cond:normality} enables us to replace the continuum to $\infty$ by a continuum to $\{z:|z|=R\}$ in the final stage of the proof where we apply the Milloux--Schmidt inequality; we omit the details.

Note that the condition \eqref{cond:normality} implies that $(F_n)$ forms a normal family by using a form of Montel's theorem due to Carath\'eodory and Landau; see \cite[p.~202]{Caratheodory}.

Finally, Theorem~F holds with no modifications when the Euclidean metric is replaced by the spherical metric, since the deduction of part~(b) of Theorem~\ref{thm:Pomgen} from part~(a) is essentially unchanged in the spherical setting.

\begin{rem*}
As this paper was being completed we learnt of very interesting recent work by Mart\'i-Pete, Rempe and Waterman in~\cite{MRW}. Of relevance to our current paper, ~\cite[Theorem~1.10]{MRW} concerns the dynamical behaviour of boundary points of a wandering domain~$U$ of an entire function. They prove that the set of points $z\in \partial U$ such that $\limsup_{n\to\infty}\chi(f^n(z),f^n(z_0))> 0$, where $z_0\in U$, which they call {\it maverick points}, forms a set of harmonic measure zero. In the setting of wandering domains and the spherical metric, their result is stronger than the results given in this section, though the results in this section hold in our more general setting.  Our results in earlier sections capture the different possible behaviours of orbits of boundary points tending to infinity.
\end{rem*}

\section{Open questions}\label{questions}
In this final section we discuss several interesting questions, which arise in connection with our new results. The first relates to a possible generalisation of the ADM dichotomy, part~(b), discussed in Section~\ref{ADMproof}. Recall that the Denjoy--Wolff set is only defined if interior orbits converge to the boundary, and that it follows from Theorem~A that, in this case, all interior orbits have the same limiting behaviour. For simplicity, the question is stated for inner functions; in a more general version, the geometry of the domain's boundary would play a role.
\begin{qn}\label{qn1}
(a) Must the Denjoy--Wolff set have zero measure for any contracting forward composition of inner functions $F_n=f_n\circ \cdots \circ f_1$, $n\in\N$, such that interior orbits converge to the boundary sufficiently slowly that $\sum_{n\in\N} (1-|F_n(0)|)$ is divergent?

(b) More generally, we can ask this question for contracting sequences $(F_n)$ of inner functions such that interior orbits converge to the boundary sufficiently slowly that this series is divergent.
\end{qn}

Including the contracting hypothesis in Question~\ref{qn1} is in some sense natural since for iteration of a single inner function the divergence condition implies the contracting property; see the discussion in Section~\ref{ADMproof}.

If the answer to Question~\ref{qn1}, part~(b) does turn out to be `yes', this would imply that a version of Theorem~F holds for sequences of inner functions, not just forward compositions of them, since if the sum is convergent we can apply Theorem~B, which gives full measure for the Denjoy--Wolff set.

Next we have a couple of questions relating specifically to our examples. In connection with Theorem~D, concerning a sequence of M\"obius self-maps of $\D$ with empty Denjoy--Wolff set, the following question arises.

\begin{qn}\label{qn2}
Does there exist an example of a forward composition of non-M\"obius self-maps of $\D$ for which interior orbits converge to the boundary and the Denjoy--Wolff set is empty? Similarly, does there exist an orbit of wandering domains, univalent or not, for which the interior orbits converge to the boundary and the Denjoy--Wolff set is empty?
\end{qn}

Another of our examples suggests a further question.
\begin{qn}\label{qn3}
In Example~\ref{Bn-to-1}, in which an entire boundary neighbourhood of~1 converges to~1, can we deduce that the set of boundary points which converge to~1 is dense in $\partial \D$ or even has full measure there?
\end{qn}

Finally, we learnt the following  related and interesting question from Marco Abate: it is known that there is no straightforward generalisation of the Denjoy--Wolff theorem from $\D$ to a general simply connected domain $U$; see \cite[Problem~5-a]{milnor06} for an example of a comb domain with non-locally connected boundary within which orbits of a univalent self-map accumulate at a continuum in the boundary.
\begin{qn}\label{qn6}
Let~$f$ be a holomorphic self-map of a simply connected domain $U$. What non-trivial conditions on~$f$ and~$U$ are sufficient to ensure that a Denjoy--Wolff point exists in~$\ov{U}$ for~$f$?
\end{qn}
First note that our Theorem~A implies that if $f$ is a holomorphic self-map of a simply connected domain~$U$ and one orbit converges to a point $p\in\partial U$, then all others must do the same and hence~$p$ is the unique Denjoy--Wolff point.

A trivial sufficient condition for Question~\ref{qn6} is that~$U$ is a Jordan domain, and one might conjecture that if~$f$ has a full radial extension to $\partial U$, possibly mapping $\partial U$ to itself almost everywhere, then this property would also be sufficient. However, it is straightforward to modify the example in \cite[Problem~5-a]{milnor06}, replacing the teeth of the comb by thin triangles, to ensure that the univalent self-map has a full radial extension. Moreover, the univalent self-map of such a domain~$U$ can be replaced by a self-map of the domain that is conjugate via a Riemann map to a degree 2 hyperbolic Blaschke product with an attracting fixed point on the boundary, and within~$U$ orbits under this self-map also accumulate at a continuum.

\begin{rem*}
There are versions of the Denjoy--Wolff theorem for self-maps of hyperbolic Riemann surfaces in which the boundary limit set is in general a single point or a continuum; see \cite{Heins1} and \cite{Heins2}. See also \cite{Short-Christ} for results concerning the stability of the Denjoy--Wolff theorem under small perturbations of a holomorphic self-map of $\D$ and \cite{Abate-Chr} for generalisations of such results to self-maps of hyperbolic Riemann surfaces.
\end{rem*}

\vspace*{10pt}

\bibliography{Wandering2}

\hspace{-0.3cm}{\em Emails}:\\ambenini@gmail.com,
 vasiliki.evdoridou@open.ac.uk,  nfagella@ub.edu, phil.rippon@open.ac.uk, gwyneth.stallard@open.ac.uk.

\hspace{-0.3cm}{\em ORCID iDs}:\\
Anna Miriam Benini: 0000-0001-9014-5235\\
Vasiliki Evdoridou: 0000-0002-5409-2663\\
N\'uria Fagella: 0000-0002-5466-0579\\
Phil Rippon: 0000-0003-4844-1190\\
Gwyneth Stallard: 0000-0001-7621-2219

\end{document}